\documentclass[11pt]{amsart}
\usepackage{amssymb}
\usepackage{amscd}
\usepackage{mathrsfs}
\usepackage{xy}
\usepackage{hyperref}
\usepackage[english]{babel}

\numberwithin{equation}{section}

\newtheorem{mthm}{Theorem}

\newtheorem*{mvar}{Variation}
\newtheorem{thm}{Theorem}[section]
\newtheorem{cor}[thm]{Corollary}
\newtheorem{prop}[thm]{Proposition}
\newtheorem{lem}[thm]{Lemma}

\theoremstyle{definition}

\newtheorem{rem}[thm]{Remark}
\newtheorem*{rem*}{Remark}
\newtheorem{rems}[thm]{Remarks}

\def\one{\mathbf{1\kern-1.6mm 1}}

\def\se{\subseteq}
\def\d{\,\mathrm{d}}
\def\cat#1{\ensuremath{\mathrm{CAT}(#1)}}
\def\CC{\mathbf{C}}

\def\NN{\mathbf{N}}
\def\PP{\mathbf{P}}
\def\RR{\mathbf{R}}
\def\ZZ{\mathbf{Z}}
\def\QQ{\mathbf{Q}}
\def\TT{\mathbf{T}}
\def\Ker{{\rm Ker}}

\def\Aut{{\rm Aut}}

\def\GL{{\rm GL}}
\def\SL{{\rm SL}}
\def\SO{{\rm SO}}

\def\Ramen{{\rm Ramen}}

\def\Homeo{{\rm Homeo}}
\def\Prob{{\mathrm{Prob}}}

\def\O{\mathscr{O}}
\def\ro{\varrho}
\def\fhi{\varphi}
\def\teta{\vartheta}

\def\lra{\longrightarrow}
\def\No{N\raise4pt\hbox{\tiny o}\kern+.2em}
\def\no{n\raise4pt\hbox{\tiny o}\kern+.2em}

\def\beq{\begin{equation}}
\def\eeq{\end{equation}}

\def\acts{\curvearrowright}

\newcommand\ip[2]{\left\langle{#1},{#2}\right\rangle}
\def\normal{\triangleleft}
\DeclareMathOperator*{\esssup}{ess\,sup}

\title{Product groups acting on manifolds}
\author[Alex Furman]{Alex Furman*}
\address{University of Illinois at Chicago}
\email{furman@math.uic.edu}
\thanks{*Supported in part by NSF grants DMS-0094245, DMS-0604611. }
\author[Nicolas Monod]{Nicolas Monod**}
\address{University of Chicago \& Universit\'e de Gen\`eve}
\curraddr{EPFL, Switzerland}
\email{nicolas.monod@epfl.ch}
\thanks{**Supported in part by the Swiss National Science Foundation}
\begin{document}
\begin{abstract}
We analyse volume-preserving actions of product groups on Riemannian manifolds. To this end, we establish a new superrigidity theorem for ergodic cocycles of product groups ranging in linear groups. There are no \emph{a priori} assumptions on the acting groups, except a spectral gap assumption on their action.

Our main application to manifolds concerns irreducible actions of Kazhdan product groups. We prove the following dichotomy: Either the action is infinitesimally linear, which means that the derivative cocycle arises from unbounded linear representations of all factors. Otherwise, the action is measurably isometric, in which case there are at most two factors in the product group.

As a first application, this provides lower bounds on the dimension of the manifold in terms of the number of factors in the acting group. Another application is a strong restriction for actions of non-linear groups. 
\end{abstract}
\maketitle
\section{Introduction and Statement of the Main Results}
\subsection{Actions on Manifolds}
Consider a group $\Gamma$ acting on a compact Riemannian manifold $M$
by volume-preserving diffeomorphisms. What restrictions, if any, does
the structure of $\Gamma$ impose upon the dimension of $M$ and the
dynamics of the action?

\smallskip

When $\Gamma$ is a lattice in a semi-simple group of higher rank, this
investigation is the object of Zimmer's programme~\cite{Zimmer:ICM},
which aims at a non-linear (or infinite-dimensional) extension of Margulis'
work~\cite{Margulis:ICM},\cite{Margulis:book} on superrigidity.
(See \emph{e.g.}~\cite{Zimmer:cocyclesuper:80}, 
\cite{Zimmer:book:84}, 
\cite{Zimmer:Tonmanifolds:84},
\cite{Zimmer:IHES:84}, 
\cite{LubotzkyZimmer:Topology:01},
\cite{FisherZimmer:02}, 
\cite{FisherMargulis}.)
Zimmer conjectures that there is no (infinite) action when the manifold has lower dimension
than any linear representation of the semi-simple group. In this setting, Zimmer's cocycle superrigidity
theorem establishes the existence of an invariant measurable metric. Had it been a smooth metric,
the original action would have been smoothly conjugated to translations on a homogeneous space of
a compact group, thus answering the question. In~\cite{Zimmer:spectrum:91}, Zimmer
provides a \emph{measurable} conjugation using Kazhdan property~(T).

\bigskip

In this paper, we shall focus on the case where $\Gamma=\Gamma_{1}\times\cdots\times\Gamma_{n}$
is a product of (at least two) groups or perhaps a lattice in suitable product groups.
This apparently weak hypothesis has recently been found to be
a fertile ground for very diverse rigidity phenomena
(see \emph{e.g.}~\cite{BFGM:FTLp:05}, 
\cite{BaderFurmanShaker}, 
\cite{BaderShalom:Factor:05},
\cite{BurgerMonod:GAFA:02}, 
\cite{Burger-Mozes2},
\cite{Gelander-Karlsson-Margulis}, 
\cite{HjorthKechris:MAMS:05},
\cite{Monod:SuperSplitting:06}, 
\cite{MonodShalom:CO:04},
\cite{MonodShalom:OE:05}, 
\cite{Popa07JAMS}, 
\cite{Shalom:Inven:00}).

\smallskip

One should of course discard the case where the action is a
combination of independent actions of the factors, for instance
when $M$ is a product $M_{1}\times\cdots\times M_{n}$ with 
component-wise action. We shall impose the following condition, hereafter called
\emph{ergodic irreducibility}: Each factor $\Gamma_{i}$ act ergodically on $M$. 

\medskip

Very specific examples of ergodically irreducible actions are provided by algebraic actions
on homogeneous manifolds through linear representations of the acting groups
(see Section~\ref{sec:further} below for a description).

The main motivation of this paper is to show that \emph{from the measurable cocycle viewpoint}
these linear examples are in fact essentially the only possibilities when
the groups have Kazhdan's property~(T). More precisely:
\begin{mthm}
\label{T:product-of-T}
Let $\Gamma=\Gamma_{1}\times \cdots \times \Gamma_{n}$ be a product of $n\ge 2$ groups
with property~(T) with a smooth volume-preserving action on a compact Riemannian
manifold $M$ (of non-zero dimension) such that each $\Gamma_{i}$ acts ergodically.
Then, either:
\begin{itemize}
\item[(${\rm Lin}_{1}$)]
There are for each $i$ finite index subgroups $\Gamma_{i}^*<\Gamma_{i}$
with Zariski-dense representations $\Gamma_{i}^*\to H_{i}$
in connected non-compact semi-simple Lie groups $H_{i}$;
\item[(${\rm Lin}_{2}$)]
The product $H=H_{1}\times\cdots\times H_{n}$ is a subquotient
of $\SL_{\dim(M)}(\RR)$; in particular, $\dim(M)\ge 3\,n$;
\item[(${\rm Lin}_{3}$)]
The derivative cocycle of $\Gamma\acts M$ is virtually semi-conjugated modulo
an amenable group to the product representation $\prod_i \Gamma^*_i\to H$;
\end{itemize}
Or otherwise, the following hold:
\begin{itemize}
\item[(${\rm Ism}_{1}$)]
There are only two factors: $\Gamma=\Gamma_{1}\times\Gamma_{2}$;
\item[(${\rm Ism}_{2}$)]
There are homomorphisms $\tau_{i}:\Gamma_{i}\to K$ with dense images in some compact group $K$.
\item[(${\rm Ism}_{3}$)]
The $\Gamma$-action on $M$ is measurably isomorphic to the bilateral action on $K$:
$$(g_{1},g_{2}) k\ =\ \tau_{1}(g_{1})\,k\,\tau_{2}(g_{2})^{-1}.$$
\end{itemize}
\end{mthm}
\begin{rems}\label{R:pesin}
(1) In case~(${\rm Lin}$), the most visible consequence of our statement
is perhaps that each factor $\Gamma_i$ of the group $\Gamma$ must admit 
an unbounded real linear representation, and the restriction on the dimension.
As for the dynamical conclusion~(${\rm Lin}_3$), it will be made more precise
below; it involves a sequence of reductions such as passing to an algebraic hull
and dividing by the amenable radical. This procedure still retains some basic characteristics
of the derivative cocycle which enables one to describe, for example, the Kolmogorov--Sinai entropy $h(g,M)$
of individual diffeomorphisms $g$ of $M$:
\begin{equation}\label{E:pesin}
h(g,M)\ =\ \sum_k m_k \lambda_k \qquad\big(g\in\prod_i \Gamma^*_i\big),
\end{equation}
where the $\lambda_k$ are the positive Lyapunov exponents of $\ro(g)$ and $m_k$ the corresponding
total multiplicities as explained in Section~\ref{S:cors}.

(2) The proof of Theorem~\ref{T:product-of-T} combines our superrigidity theorem (Theorem~\ref{T:cocycle2}
below) with Zimmer's~\cite{Zimmer:spectrum:91}. In our setting, the (Ism) alternative takes a particularly
precise form (only two factors, no isotropy subgroup of the compact group). The problem of the smoothness
of the conjugation remains.
\end{rems}
Theorem~\ref{T:product-of-T} is a rigidity statement describing suitable actions of product
groups as arising from linear representations and in particular providing dimension bounds.
As an immediate by-product, it rules out certain actions, see Section~\ref{sec:further}.
For non-Kazhdan groups, a number of the conclusions still hold as long as at least one factor has
property~(T):

\begin{mvar}[On Theorem~\ref{T:product-of-T}]
Let $\Gamma=\Gamma_{1}\times \cdots \times \Gamma_{n}$ be a product of $n\ge 2$ groups
with a smooth volume-preserving action on a compact Riemannian
manifold $M$ (of non-zero dimension) such that each $\Gamma_{i}$ acts ergodically.
Assume that $\Gamma_i$ has property~(T) for $1\leq i \leq k$, where $k\geq 1$.
Then, either:
\begin{itemize}
\item[(${\rm Lin'}_{1}$)]
For $i\leq k$, there are finite index subgroups $\Gamma_{i}^*<\Gamma_{i}$
with Zariski-dense representations $\Gamma_{i}^*\to H_{i}$
in connected non-compact semi-simple Lie groups $H_{i}$;
\item[(${\rm Lin'}_{2}$)]
The product $H=H_{1}\times\cdots\times H_{k}$ is a subquotient
of $\SL_{\dim(M)}(\RR)$; in particular, $\dim(M)\ge 3\,k$;
\item[(${\rm Lin'}_{3}$)]
The derivative cocycle of $\Gamma_1\times \cdots \times \Gamma_k \acts M$ is virtually semi-conjugated modulo
an amenable group to the product representation $\prod_{i\leq k} \Gamma^*_i\to H$;
\end{itemize}
Or otherwise, $({\rm Ism}_{1,2,3})$ hold as in Theorem~\ref{T:product-of-T}.
\end{mvar}

The possibility of non-Kazhdan factors actually restricts further the alternative~(Ism), since for instance a compact group
cannot contain both an infinite Kazhdan group and a dense commutative (or more generally soluble) subgroup. For example:

\smallskip
\itshape
Let $\Gamma=\Gamma_{1} \times \Gamma_{2}$ with $\Gamma_1$ non-linear Kazhdan and $\Gamma_2$ soluble.
Then $\Gamma$ admits no ergodically irreducible smooth volume-preserving action on any compact Riemannian
manifold of non-zero dimension.
\upshape

\bigskip

We now turn to actions of groups $\Gamma$ which, rather than being products, are lattices in
a product $G=G_1\times \cdots \times G_n$ of $n\geq 2$ locally compact groups. Such lattices are assumed
\emph{irreducible} in the sense that $G_i\cdot\Gamma$ is dense in $G$ for all $i$.

If the lattice is not cocompact, we shall need to assume it to be \emph{integrable};
this condition will be explained in more detail below and means that the canonical cocycle
class $G\times G/\Gamma\to\Gamma$ has a representative $c$ such that the word-length of $c(g, -)$
is in $L^1(G/\Gamma)$ for every $g\in G$.

\begin{mthm}[Lattices in Products]
\label{T:lattices}
Let $G=G_{1}\times\cdots\times G_{n}$ be a product of $n\ge 2$
locally compact second countable groups with property~(T).
Let $\Gamma<G$ be a cocompact or integrable irreducible lattice with a 
mixing smooth volume-preserving action a compact Riemannian manifold $M$.

Then $G$ admits an unbounded continuous real linear representation.

Moreover, this $G$-representation is virtually semi-conjugated modulo an amenable group
to the cocycle induced to $G$ from the derivative cocycle of $\Gamma\acts M$.
\end{mthm}

The \emph{induced} cocycle in the statement refers to the following construction.
Let $G^*< G$ be a closed subgroup such that $G/G^*$
carries a $G$-invariant probability measure (\emph{e.g.} $G^*$ is a lattice in $G$,
or $G^*$ is closed of finite index). The \emph{canonical class}
is realized by the cocycles
\[
        c:G\times G/G^*\lra G^*,\qquad c(g,hG^*)=\sigma(ghG^*)^{-1}g\sigma(hG^*)
\]
where $\sigma:G/G^*\to G$ is a measurable cross-section;
different choices of $\sigma$ give rise to cohomologous cocycles.
To any $G^*$-space $X$ corresponds the \emph{induced} $G$-space $G/G^* \ltimes X$
with product measure and action $g(h G^*, x)=(gh G^*, c(g, h G^*)x)$. To any
cocycle $\alpha: G^*\times X\to H$, \emph{e.g.} a homomorphism $G^*\to H$, corresponds 
the \emph{induced cocycle}
$$G\times (G/G^* \ltimes X) \lra H,\qquad
(g,(h G^*, x)) \longmapsto \alpha(c(g, h G^*),x).$$
%

\subsection{Cocycle Superrigidity}
The main new tool in the proofs of the above results 
is a new cocycle superrigidity result for Lie-group-valued cocycles
of ergodically irreducible actions of product groups on probability spaces.
This result does not use property~(T) of the acting group, but the weaker
property of \emph{spectral gap} for the action.
Our cocycle superrigidity theorem also requires some integrability or boundedness condition
on the values of the cocycle, which can always be assumed
for derivative cocycles arising from actions on compact manifolds.

\smallskip

We first state a simpler version of this result, although for
the above applications we shall need the more general statement of
Theorem~\ref{T:cocycle2} below.
\begin{mthm}[Cocycle Superrigity --- Semi-Simple Hull]
\label{T:cocycle1}
Let $G=G_{1}\times\cdots \times G_{n}$ be a product of $n\ge 2$ locally compact second countable
groups with a measure-preserving action on a standard probability space 
$X$. Let $\alpha:G\times X\to H$ be a measurable cocycle
ranging in a connected centre-free semi-simple Lie group $H$
without compact factors. Assume:
\begin{itemize}
\item[{(Zd)}]
        $\alpha$ is Zariski-dense in $H$, i.e. it is not cohomologous 
        to a cocycle ranging in a proper algebraic subgroup of $H$;
\item[{(SG)}]
        Each of the actions $G_{i}\acts X$ is ergodic
        and has a spectral gap;
\item[($L^{1}$)]
        $\log\|\alpha(g,-)\|\in L^{1}(X)$ for very $g\in G$.
\end{itemize}
Then $\alpha$ is cohomologous to a continuous homomorphism $\ro: G\to H$.
\end{mthm}

\begin{rem}\label{R:details}
The assumption that $H$ has no compact factors is necessary, see Section~\ref{S:compact}.
On the other hand, our proof shows that one can drop the spectral gap assumption for
one of the factors $G_i$. This is also the case for the general form in Theorem~\ref{T:cocycle2} below.
\end{rem}

Our proof of Theorem~\ref{T:cocycle1} uses notably random walks and Oseledets'
theorem (an influence of Margulis' work~\cite{Margulis:announced:74},\cite{Margulis:ICM}).
In Section~\ref{S:splitting}, we also point to an alternative geometric proof using
the \cat0 techniques of~\cite{Monod:SuperSplitting:06}. We also mention that
S.~Popa established very general cocycle superrigidity results for products with spectral
gap~\cite{Popa07JAMS}. A very striking feature of his results is that they have no restriction
at all on the target groups. On the other hand, they are about a specific class of actions
and therefore cannot be used for our present purposes.

\medskip

The general result (Theorem~\ref{T:cocycle2} below) 
does not assume Zariski-density of the cocycle.
Therefore, its statement requires that we recall a few other basic constructions.  

\medskip

\noindent
\textbf{Radical.}
For any topological group $L$ one can define the \emph{amenable radical} 
$\Ramen(L)$ to be the maximal closed normal amenable subgroup in $L$
(in analogy to Zassenhaus' classical definition~\cite{Zassenhaus}).
If $L$ is a connected real algebraic group, then the quotient $L/\Ramen(L)$ 
can be obtained by dividing first by the soluble radical of $L$, 
then by the centre of the resulting reductive group, 
and then by the product of all the compact factors 
of the resulting connected semi-simple centre-free Lie group.

\medskip

\noindent
\textbf{Semi-conjugacy.}
Let $\alpha:G\times X\to V$ be a cocycle. If some conjugate of $\alpha$
ranges in a subgroup $L<V$ and $p:L\to H$
is (the canonical morphism to) a quotient group, consider the corresponding cocycle $\teta: G\times X\to H$.
We shall say that $\alpha$ is \emph{semi-conjugated} to $\teta$ modulo $\Ker(p)$.

\medskip

\noindent
\textbf{Hull.}
For a cocycle $\alpha:G\times X\to V$ ranging in an algebraic group $V$
(over $\RR$ in our case) there is a \emph{minimal} algebraic subgroup $L<V$
into which some conjugate $\beta=\alpha^{f}:G\times X\to L<V$ ranges;
if $G\acts X$ is ergodic,
this group $V$ is unique up to conjugacy and is called the
\emph{algebraic hull}~\cite[9.2]{Zimmer:book:84}.
The neutral component (connected component of the identity)
$L^{0}$ of $L$ is an algebraic subgroup of finite 
index in $L$. There is an ergodic finite extension $\pi:Y\to X$ of the original action
such that the lift $\tilde\beta:G\times Y\to L$ of $\beta$ is cohomologous
to a cocycle into $L^{0}$ (this construction from~\cite[9.2.6]{Zimmer:book:84}
is recalled in the proof of Theorem~\ref{T:cocycle2} below).

\medskip

We are ready to state the general result; for the application to Theorem~\ref{T:product-of-T},
one can assume that $G$ is discrete.
\begin{mthm}[Cocycle Superrigidity --- Unrestricted Hull]
\label{T:cocycle2}
Let $G=G_{1}\times \cdots \times G_{n}$ be a product of $n\ge 2$
locally compact second countable groups with a measure-preserving action on a
standard probability space $X$.
Let $\alpha:G\times X\to \GL_d(\RR)$ be a measurable cocycle. Assume:
\begin{itemize}
\item[{(SG)}]
        Each of the actions $G_{i}\acts X$ is ergodic
        and has a spectral gap;
\item[($L^{\infty}$)]
        $\|\alpha(g,-)\|\in L^{\infty}(X)$ for every $g\in G$.
\end{itemize}
Let $L$ be the algebraic hull of $\alpha$, 
$L^{0}\normal L$ the neutral component, and
$p:L^{0}\to H=L^{0}/\Ramen(L^{0})$ the quotient by the amenable radical.
Then we have:
\begin{enumerate}
\item
A splitting $H=H_{1}\times\cdots \times H_n$ into (possibly trivial)
connected centre-free semi-simple real Lie groups without compact factors;
\item
Finite index open subgroups $G_{i}^*< G_{i}$ and continuous homomorphisms
$\ro^*_{i}:G_{i}^*\to H_{i}$ with Zariski-dense image.
\end{enumerate}
Upon lifting $\alpha$ to a finite ergodic $G$-extension
$Y\to X$, it is semi-conjugated to the cocycle induced from the product
representation $\prod_i \ro^*_i:G^*\to H$ of $G^*=\prod_i G^*_i$.
\end{mthm}

(The conclusion of this theorem is a precise formulation of what we called \emph{virtual semi-conjugacy} in Theorem~\ref{T:product-of-T} above.)

\subsection{Further comments}\label{sec:further}
There are three basic algebraic examples of volume-preserving, ergodically irreducible actions on homogenous
compact manifolds \emph{via} linear representations: semi-simple, nilpotent and compact.

\medskip
\noindent
\textbf{Semi-simple.}
Let $\Gamma=\Gamma_{1}\times\cdots\times\Gamma_{n}$,
where each $\Gamma_{i}$ admits a representation $\Gamma_{i}\to H_{i}$ 
with unbounded image in a real semi-simple Lie group $H_{i}$,
and let $H=H_{1}\times\cdots\times H_{n}$.
Embed $H$ into some real semi-simple 
Lie group $L$ (\emph{e.g.} $L=\SL_{N}(\RR)$ with $N\ge\sum \dim Lie(H_{i})$)
and choose an irreducible cocompact lattice $\Lambda$ in $L$.
Then the $\Gamma$-action on $M=L/\Lambda$ by left translations \emph{via} $\Gamma\to H<L$
is ergodically irreducible, and even mixing (Howe-Moore theorem).

\medskip
\noindent
\textbf{Nilpotent.}
Let $N$ be a connected real nilpotent group with a lattice $\Lambda<N$;
such lattices are always cocompact and arithmetic by a theorem of Mal'cev, see \emph{e.g.}~\cite{Malcev51}
and~\cite{Raghunathan}. If a group $\Gamma$ admits a homomorphism $\rho:\Gamma\to\Aut(N)$
where $\rho(\Gamma)$ normalizes $\Lambda$, one obtains a volume-preserving $\Gamma$-action
on the compact nil-manifold $M=N/\Lambda$. 
The standard $\SL_d(\ZZ)$-action on the torus $\TT^d$ is a prime example of this setting, and indeed it is the critical
case to investigate ergodic irreducibility since any compact nil-manifold admits such a torus as an equivariant quotient.
Let thus $\Gamma=\Gamma_{1}\times\cdots\times\Gamma_{n}$ with infinite homomorphims $\rho_i:\Gamma_i\to \SL_{d_i}(\ZZ)$.
Upon passing to finite index subgroups and reducing the dimensions $d_i$ one may assume that 
$\ro_i(\Gamma_i^*)$ is irreducible over $\QQ$ for all finite index subgroups $\Gamma_i^*<\Gamma_i$.
Let $\ro=\ro_1\otimes\cdots\otimes \ro_n$ be the tensor representation taking values
in $\SL_d(\ZZ)$ where $d=\prod d_i$.
Then the corresponding $\Gamma$-action on $M=\TT^d$ is ergodically irreducible.
Indeed, a Fourier transform argument shows that the ergodicity of $\Gamma_i$ is equivalent
to $\ro(\Gamma_i^*)$ having no invariant vectors 
in $\QQ^d\setminus\{0\}$ for finite index subgroups $\Gamma_i^*<\Gamma_i$, which follows by construction.
Note that the dimension $d=\dim(M)$ will typically exceed that of semi-simple examples
and in addition the required linearity assumption is more stringent: $\Gamma_i$ must have unbounded 
representations defined over $\ZZ$, rather than over $\RR$. 

\medskip
\noindent
\textbf{Compact (isometric).}
Let $\Gamma=\Gamma_{1}\times\Gamma_{2}$ and suppose there are
homomorphisms $\tau_{i}:\Gamma_{i}\to K$ with dense images in some compact group $K$.
Then the $\Gamma$-action on $M=K$ given by
$$(g_{1},g_{2}) k\ =\ \tau_{1}(g_{1})\,k\,\tau_{2}(g_{2})^{-1}
        \qquad (g_{i}\in\Gamma_{i},\ k\in K)$$
is ergodically irreducible. If $K$ is moreover a connected Lie group, we have
an ergodically irreducible volume-preserving action on a manifold, and this action is even isometric.

\bigskip

An example where Theorem~\ref{T:product-of-T} rules out altogether certain actions is as follows.

\begin{cor}\label{C:non-linear}
Let $\Gamma=\Gamma_{1}\times\cdots\times\Gamma_{n}$ be a product of $n\geq 2$
infinite groups with property~(T), where one of the factors 
does not admit unbounded real linear representations.

Then $\Gamma$ has no smooth mixing volume-preserving actions on compact manifolds.
\end{cor}

Mixing is a convenient strengthening of ergodic irreducibility, but the latter
more natural assumption suffices in various cases. For instance, the statement of
Theorem~\ref{T:product-of-T} shows that ergodic irreducibility suffices in
Corollary~\ref{C:non-linear} if in addition $n\geq 3$, or if $\Gamma_{1}$ and
$\Gamma_{2}$ cannot not embed densely in the same compact group. Further:

\begin{cor}\label{C:no-actions}
Let $\Gamma=\Gamma_{1}\times\cdots\times\Gamma_{n}$ be a product of $n\geq 2$
groups with property~(T), where one of the factors admits only finitely many finite quotients.

Then $\Gamma$ has no ergodically irreducible smooth volume-preserving
actions on compact manifolds of non-zero dimension.
\end{cor}

We recall in Section~\ref{S:cors} that there are many infinite groups with Kazhdan's property~(T) that do not admit
unbounded linear representations over $\RR$ and some are known not to have finite quotients.

\subsection*{Acknowledgments}
This article owes much to the work of H.~Furstenberg, G.~Margulis and R.~Zimmer.
We are grateful to D.~Fisher and the anonymous referee for interesting and useful remarks.

\section{General Notations}
\label{S:notations}
Throughout the paper, a \emph{lcsc group} will mean a locally compact second countable
topological group. We denote by $m_G$ a choice of left Haar measure and recall that its
measure class depends only on $G$. If $G$ is a lcsc group, a \emph{probability $G$-space}
refers to a standard (Lebesgue) probability space endowed with a measurable $G$-action that
preserves the measure. All (ergodic-theoretical) cocycles are assumed measurable. If
$\alpha:G\times X\to H$ is a cocycle and $f:X\to H$ a measurable map, the corresponding cocycle
$\alpha^f$ \emph{cohomologous} (or \emph{conjugated}) to $\alpha$ is defined by
$$\alpha^f(g, x)\ =\ f(g x)^{-1} \alpha(g, x) f(x).$$
For any subgroup $L<G$, we denote simply by $\alpha|_L$ the restriction of $\alpha$
to $L\times X$. If $Y$ is some $H$-space, one defines the \emph{skew product}
$G$-space $X\ltimes Y$ by endowing the product $X\times Y$ with the action $g(x,y)=(g x, \alpha(g,x) y)$.
Notice that \emph{induced actions}, as defined in the introduction, are a special case
of this construction.

\smallskip

A \emph{norm} on a group $H$ will mean a map $\|\cdot\|: H\to [1,
\infty)$ such that $\|gh\|\leq \|g\|\cdot\|h\|$ and $\|g^{-1}\|=\|g\|$
for all $g,h\in H$; by default, we think of $\GL_d(\RR)$ as endowed
with $\|g\|=\|g\|_\mathrm{op}\cdot\|g^{-1}\|_\mathrm{op}$, wherein
$\|\cdot\|_\mathrm{op}$ is the operator norm. (The terminology
conflicts of course with normed vector spaces.)  The following fact is well-known both
in the measurable setting for lcsc groups and for Borel norms on Baire topological groups:

\begin{lem}
\label{L:bdd}
Any measurable norm on a lcsc groups is bounded on compact sets.\hfill\qedsymbol
\end{lem}

\smallskip

We write $\GL^1_d< \GL_d$ for the subgroup of determinant $\pm 1$.

\smallskip

Let $G$ be a lcsc group and $(X,\mu)$ an ergodic probability $G$-space.
We say that the $G$-action has a \emph{spectral gap} if the $G$-representation
on
$$L^{2}_{0}(X)=L^2(X,\mu)\ominus\CC=\Big\{f\in L^2(X,\mu)\ :\ \int f\d\mu=0\Big\}$$
does not almost have invariant vectors. This representation is the (Koopman) 
representation given by $gf=f\circ g^{-1}$. Explicitly, the spectral gap means that
there is a compact subset $K\subset G$ and an $\epsilon>0$ such that
\begin{equation}
\label{e:gap}
        \forall\, F\in L^{2}_{0}(X),\ \exists\, g\in K:\qquad \|g F - F\|\ge \epsilon\cdot \|F\|.
\end{equation}
The action $G\acts X$ is called \emph{mixing} if the matrix coefficients of the representation
$L^{2}_{0}(X)$ are $C_0$ (vanish at infinity).

\smallskip

We use $\succ$ for the weak containment of unitary representations;
the trivial representation of a lcsc group $G$ is denoted by $\one_G$. Thus for instance
the above spectral gap property is $L^2_0(X)\not\succ\one_G$. This
terminology is due to the following classical characterization
\emph{\`a la} Kesten for arbitrary unitary representations $\pi$, see
\emph{e.g.}~\cite[G.4.2]{Bekka-Harpe-Valette}.

\begin{lem}\label{l:kesten}
Let $\sigma$ be a probability measure on $G$ that is absolutely continuous with respect to Haar
measures and whose support generates $G$. Then $\pi\not\succ\one_G$ if and only if
the spectral radius of $\pi(\sigma)$ satisfies $\|\pi(\sigma)\|_\mathrm{sp}<1$.\hfill\qedsymbol
\end{lem}

All ergodic-theoretical statements regarding actions on compact Riemannian manifolds 
are understood to refer to the normalised measure defined by the volume form.
We endow by default compact groups with their normalised Haar measure. More generally, when
 $H<G$ is a closed subgroup of a lcsc group such that $K/L$ admits a
 non-zero $G$-invariant measure, we endow it with such a measure which we denote by $m_{G/H}$.
We normalise $m_{G/H}$ whenever it is finite.

\section{From Cocycle superrigidity to Theorem~\ref{T:product-of-T}}
In this section, we deduce Theorem~\ref{T:product-of-T} and its corollaries
from the general cocycle superrigidity Theorem~\ref{T:cocycle2}. To this end, we use also the main result
of Zimmer's~\cite{Zimmer:spectrum:91} which
states that a smooth volume-preserving action of
a property~(T) group on a compact manifold preserving
a measurable Riemannian structure has discrete spectrum.
This result uses Kazhdan's property~(T) and smoothness
in an essential way (more than just the spectral gap for
the action on the manifold).

\subsection{Proof of Theorem~\ref{T:product-of-T}}\label{S:T:product-of-T}
A volume-preserving smooth action
$\Gamma\acts M$ defines a $\Gamma$-action
on the tangent bundle $TM$. 
The tangent bundle can be measurably trivialised, thus defining
the \emph{derivative cocycle} (unique up to cohomology) 
$\alpha:\Gamma\times M\to \GL^1_{d}(\RR)$ where $d=\dim(M)$.
Moreover, one can assume that the norm $\|\alpha(g,-)\|$ is essentially
bounded over $M$ for each $g\in\Gamma$.

To see this, choose a measurable family $\{ V\xrightarrow{p_x} U_{x}\ :\ x\in M\}$
of volume-preserving charts, where $V$ and $U_{x}$ are neighbourhoods of $0\in\RR^{d}$ 
and $x\in M$ respectively, such that $p_{x}(0)=x$
(it suffices to require $|{\rm Jac}(p_{x})(0)|=1$), with
$\|Dp_{x}(0)\|$ being uniformly bounded over
$x\in M$. This is possible by compactness of $M$.
Given such a family, one defines the measurable cocycle 
\begin{equation}
\label{e:der-coc}
        \alpha:\Gamma\times M\lra \GL^1_{d}(\RR)\qquad\text{by}
        \qquad\alpha(g,x)=D(p_{g x}^{-1}\circ g\circ p_{x})(0).
\end{equation}
Observe that changing $\{p_{x}\}_{x\in M}$ would yield cohomologous of cocycles.
Note that for each $g\in\Gamma$ the cocycle $\alpha(g,-)$ is essentially bounded, 
due to the uniform bound on $\|Dp_{x}(0)\|$
and on the derivative of $g$ over the compact manifold $M$.

\smallskip

Since each $\Gamma_{i}$ has property~(T), the ergodicity of the action $\Gamma_{i}\acts M$,
which is equivalent to the absence of $\Gamma_{i}$-invariant vectors in $L^{2}_{0}(M)$, 
yields a spectral gap in this representation. 
Hence we are in position to apply Cocycle Superrigidity
Theorem~\ref{T:cocycle2}, with $G_i$ being $\Gamma_i$ (endowed with the discrete topology)
and $X=M$.

\medskip

Suppose first that all factors $H_i$ appearing in the statement of Theorem~\ref{T:cocycle2}
are non-trivial (and thus non-compact). Then we have finite index subgroups $\Gamma_{i}^*$
with Zariski-dense representations $\ro_{i}^*:\Gamma_{i}^*\to H_{i}$ and a virtual semi-conjugacy
of the derivative cocycle to the product representation $\ro$. This is precisely the case of
the linearity phenomena (Lin) in Theorem~\ref{T:product-of-T}; we just need to justify
$\mathrm{dim}(M)\geq 3\,n$. We first observe that each $H_i$ has at least one simple factor with
property~(T). Indeed, we recall that the image of a group with property~(T) in a
semi-simple Lie group without~(T) is always bounded (and thus not Zariski-dense) since such
Lie groups have the Haagerup property~\cite{Cherix:etc}. Now it suffices to observe that
non-compact simple Lie groups with property~(T) have no unbounded real representation of dimension less
than~$3$. (This follows immediately, for instance, from the Haagerup property of $\SL_2(\RR)$ together
with the fact that Kazhdan groups have compact Abelianization~\cite{Kazhdan67}.)

\medskip

It remains to show that the only alternative is the case described in (Isom$_{1}$)--(Isom$_{3}$).
Hence, assume that at least one of the factors $H_i$ in Theorem~\ref{T:cocycle2} is compact.
In this case, at least one of the 
factors of $\Gamma$, say $\Gamma_{1}$, has the property that the restriction $\alpha|_{\Gamma_1}$
is cohomologous to a cocycle ranging in an amenable group
(namely in a finite extension of the amenable radical of $L^{0}$).

It is well-known~\cite[9.1.1]{Zimmer:book:84} that any measurable
cocycle of an ergodic action of a property~(T) group
to an amenable group is cohomologous to a cocycle ranging in some compact subgroup.
Hence the restriction $\alpha|_{\Gamma_1}$ is
cohomologous to a cocycle ranging in a compact subgroup of $\GL^1_{d}(\RR)$;
upon conjugating further, we may assume it ranges in ${\rm O}_{d}(\RR)$. 
This is equivalent to saying that $\Gamma_{1}$ preserves a measurable Riemannian
structure on $M$.

\smallskip

This is precisely a situation analysed by Zimmer in~\cite{Zimmer:spectrum:91},
where he proves (Theorem~1.7) that in such case  the action $\Gamma_{1}\acts M$
has discrete spectrum: that is, $L^{2}(M)$ splits as a direct sum of finite dimensional
$\Gamma_{1}$-invariant subspaces.
By Mackey's measure-theoretical converse to the Peter--Weyl theorem~\cite{Mackey64},
such an action is measure-theoretically isomorphic to an isometric action
(the case of a single transformation was previously established by Halmos and von
Neumann~\cite{vonNeumann32,Halmos-vonNeumann}).
This means that there exist a compact group $K$, a homomorphism $\tau_{1}:\Gamma_{1}\to K$
a closed subgroup $L<K$, and a measure space isomorphism 
\[
        T:M\xrightarrow{\,\cong\,} K/L \qquad\text{with}\qquad
        T(g_{1} x)=\tau_{1}(g_{1})T(x) 
\]
for a.e. $x\in M$ and all $g_{1}\in \Gamma_{1}$. 
We can assume that $L$ does not contain non-trivial closed subgroups that are normal in $K$
upon dividing by the kernel of the $K$-action on $K/L$.

Note that the group $N_{K}(L)/L$ acting on $K/L$  faithfully from the right
commutes with the $\Gamma_{1}$-action by left translations.
Denote by $\Aut(K/L,m_{K/L})$ the group of all measure space automorphisms,
where everything is understood modulo null sets. We recall the following easy 
\begin{lem}[{see~\cite[7.2]{Furman:outer:05}}]
The centraliser of $\tau_{1}(\Gamma_{1})$ in $\Aut(K/L,m_{K/L})$
is precisely $N_{K}(L)/L$.\hfill\qedsymbol
\end{lem}

Denote $\Gamma'_{1}=\Gamma_{2}\times\cdots\times \Gamma_{n}$, so that
$\Gamma=\Gamma_{1}\times\Gamma_{1}'$. 
By the above Lemma, the $\Gamma_{1}'$-action, which commutes 
with the $\Gamma_{1}$-action on $M\cong K/L$, 
defines a homomorphism $\tau:\Gamma_{1}'\to N_{K}(L)/L$.
Ergodicity of the $\Gamma_{1}'$-action implies that $N_{K}(L)/L$ 
also acts ergodically on $K/L$. 
But $N_{K}(L)$ and $N_{K}(L)/L$ are compact groups, 
so ergodicity means that the action is transitive, i.e., $N_{K}(L)=K$
and $L$ is normal in $K$. By our convention this means that $L$ is trivial, 
i.e. $M\cong K$. 
In particular, the representation $\tau$ of $\Gamma'_{1}$ ranges
into $K$ itself. 

\smallskip

The ergodicity assumption of the action of each $\Gamma_{i}$ on 
$M\cong K$ means that the images $\tau(\Gamma_{i})$ 
are dense in $K$.  
We claim that $n=2$, i.e. $\Gamma_{1}'=\Gamma_{2}$. 
Indeed if $n\ge 3$ then $K$ contains two commuting subgroups $\tau(\Gamma_{2})$ 
and $\tau(\Gamma_{3})$, each being dense in $K$. 
This forces $K$ to be commutative.
Property~(T) of, say $\Gamma_{1}$, implies that $\tau_{1}(\Gamma_{1})$ is finite,
hence so is $K$.
But this contradicts the measure-theoretic isomorphism of $K$ with $M$
since the volume has no atoms.
This completes the description of $\Gamma\acts M$ in this case, and thus the proof 
of Theorem~\ref{T:product-of-T}.\hfill\qedsymbol

\bigskip

For the variation on Theorem~\ref{T:product-of-T}, we recall from Remark~\ref{R:details} that cocycle
superrigidity also holds if the action of one factor lacks the spectral gap property. Therefore, we can follow
the above proof by grouping all factors without property~(T) into one factor and reason with
$$\Gamma_1\times \cdots\times \Gamma_k \times \big(\Gamma_{k+1}\times \cdots \times \Gamma_n\big).$$
Thus, the argument above can be repeated \emph{verbatim} with the only difference that the distinction of the
two cases in the alternative hinges upon whether at least one factor $H_i$ is compact \emph{for $i\leq k$}.
We emphasize that the conclusion (Ism$_1$) is still $n=2$ (rather than only $k=2$).

\subsection{Corollaries}\label{S:cors}
First we recall how to deduce the entropy formula~\eqref{E:pesin}. We refer
to~\cite{Furstenberg:afterMargulisZimmer:81} or~\cite[9.4]{Zimmer:book:84} for
more precisions on the following exposition.
For any (finite-dimensional) linear operator $A$, denote by $\{\lambda_k\}$ the set
of \emph{Lyapunov (characteristic) exponents}, that is, the set of logarithms $\log|a_j|$
of all eigenvalues $a_j$ of $A$ with
$|a_j| >1$. Thus all $\lambda_k$ are distinct and positive; the total multiplicity $m_k$ 
is the sum of the multiplicities of all $a_j$ with $\log|a_j|=\lambda_k$. One has the relation
$$\sum_k m_k \lambda_k\ =\ \max_p \lim_{n\to\infty} \frac{1}{n} \log \|\wedge^p A^n \|,$$
wherein the symbol $\wedge$ denotes the exterior $p$-power. Returning to
Remark~\ref{R:pesin} and taking $A=\ro(g)$, we observe that the quantity
$$\max_p \lim_{n\to\infty} \frac{1}{n} \log \|\wedge^p\ro(g^n) \|$$
is well-defined and is not affected by the (virtual) semi-conjugacy
that Theorem~\ref{T:product-of-T} produces.
On the other hand, Pesin's formula~\cite{Pesin} gives
$$h(g,M)\ =\ \max_p \lim_{n\to\infty} \frac{1}{n} \log \|\wedge^p D g^n|_x \| \qquad(\text{a.e. } x\in M)$$
and hence~\eqref{E:pesin} follows as claimed. For more on
convergence of the above limit and characteristic exponents, see Section~\ref{S:RW}.

\medskip

Before continuing, we recall the elementary fact that if a group has a finite index subgroup
that admits a linear representation with infinite (or unbounded) image, then so does
the initial group.

\begin{proof}[Proof of Corollary~\ref{C:non-linear}]
The statement follows immediately from Theorem~\ref{T:product-of-T}
if one recalls that mixing implies ergodic irreducibility
(since the factors are assumed infinite)
and precludes discrete spectrum.
\end{proof}

Corollary~\ref{C:no-actions} hinges on the well-known fact
that a finitely generated group is residually finite if and only if is has
an injective morphism to a compact group, see the proof:

\begin{proof}[Proof of Corollary~\ref{C:no-actions}]
It suffices to show that both conclusions offered by Theorem~\ref{T:product-of-T}
are incompatible with our assumptions. We can suppose that $\Gamma_1$
has only finitely many finite quotients.
Recall that property~(T) implies that $\Gamma_1$ is finitely generated~\cite{Kazhdan67}.
Since finitely generated linear groups are residually finite~\cite{Malcev40},
every linear image of $\Gamma_1$ is finite. This already rules out case~(${\rm Lin}$).

We consider now a homomorphism $\tau_1: \Gamma_1\to K$ as in case~(${\rm Ism}$)
and seek a contradiction. By the Peter--Weyl theorem, $K$ embeds into a product
$\prod_{n\in \NN} U_n$ of (finite-dimensional) unitary groups $U_n$. By the above discussion,
the image of $\tau_1(\Gamma_1)$ in each $U_n$ is finite, and hence by density the same
statement holds for $K$. Thus $K$ is profinite; this implies that the image of $\Gamma_1$ in
$K$ is finite. Thus $K$ is a finite group, contradicting the fact that $M$ has positive dimension
and is measurably isomorphic to $K$.
\end{proof}

Finally, we briefly recall that there are many known infinite groups with Kazhdan's property~(T) that do not admit
unbounded linear representations over $\RR$.
These include: (a)~Quotients (by infinite kernels) of
lattices in ${\rm Sp}_{n,1}(\RR)$, in view of Corlette's superrigidity~\cite{Corlette}; (b)~Lattices
in semi-simple groups of higher rank over non-Archimedean fields, by Margulis'
superrigidity~\cite{Margulis:book}; (c)~Suitable Kac--Moody
groups~\cite{RemyERN,Remy-Bonvin,Caprace-Remy06,Dymara-J02};
(d)~Gromov's constructions of simple groups or torsion groups
as quotients of arbitrary non-elementary hyperbolic groups~\cite{Gromov87}, which are Kazhdan as soon
as the corresponding hyperbolic group is so; (e)~Gromov's random groups~\cite{GromovRANDOM}. Some of the
groups listed under~(c) and~(d) above have no non-trivial finite quotients.

Of course it is expected that there are many more property~(T) groups outside the linear realm. In
fact, several of the above examples have much stronger properties than those needed for
Corollary~\ref{C:non-linear} and are therefore not really an illustration of our results.
It has been observed on several occasions (we learned it from Sh.~Matsumoto) that the
random groups of~\cite{GromovRANDOM} do not have mixing smooth volume-preserving actions
on compact manifolds. Indeed, they have the fixed point property for isometric actions on
``induced'' spaces of the form
$$\int_M \GL(T_x M) / \mathrm{O}(g_x) \d x,$$
where $g$ is the Riemannian metric; see \emph{e.g.}~\cite[Ex.~47]{Monod:SuperSplitting:06} and compare
Section~\ref{S:splitting} below.%
\footnote{D.~Fisher has informed us that upon reading the preprint version of this paper he made the following observation with
L.~Silberman~\cite{Fisher-Silberman}: if a group fixes a point in the above Hilbert manifold
and has no finite quotient, then the underlying action on the manifold is trivial. (Indeed, one combines as above Zimmer's
result~\cite{Zimmer:spectrum:91} with Peter--Weyl and Mal'cev.) The authors of~\cite{Fisher-Silberman} then give many very
interesting illustrations of this statement.}
Groups as in~(a) also have strong restrictions on their cocycles, see Fisher--Hitchman~\cite{Fisher-Hitchman:IMRN}.

\section{Preliminaries for the Cocycle Superrigidity Theorem}
We recall that a group action on a standard Borel space is \emph{tame}
(or ``smooth'' in~\cite{Zimmer:book:84}) if the quotient Borel structure
is countably separated. For continuous actions of lcsc groups
on separable metrisable spaces, this is equivalent to having locally closed orbits
by the Effros--Glimm theorem~\cite[2.1.14]{Zimmer:book:84}. We shall repeatedly use the following 
fundamental facts.

\begin{thm}\label{thm:tame:alg}
Let $G<\GL_N(\RR)$ be a real algebraic group.
\begin{itemize}
\item[(i)] If $H,L<G$ are algebraic subgroups,
then the $H$-action on $G/L$ is tame. 

\item[(ii)] The $G$-action on $\Prob(\PP^{N-1})$ is tame.
\end{itemize}
\end{thm}

\begin{proof}[On the proof]
The first statement is apparently an unpublished result of
Chevalley from the 1950s, see the introduction of~\cite{Dixmier57};
for the proof, see~\cite{Dixmier66} and~\cite[3.1.3]{Zimmer:book:84}.
The second statement is due to Zimmer (see~\cite[3.2.12]{Zimmer:book:84})
and uses a result of Furstenberg~\cite{Furstenberg:Poisson:63}.
(A statement for measurable maps was given by Margulis in~\cite{Margulis:ICM:trans},
see also~\cite[\S7]{Zimmer:cocyclesuper:80}.)
\end{proof}

The following fact adapted from~\cite{MonodShalom:CO:04} is very general; it holds
even for cocycles over not necessarily measure-preserving actions.
\begin{prop}[Cocycle Splitting]
\label{P:CRL}
Let $G=G_{1}\times G_{2}$ and $H$ be lcsc groups, 
$G\acts (X,\mu)$ a measurable measure class-preserving action
on a standard probability space and $\alpha:G\times X\to H$
a cocycle. Assume:
\begin{itemize}
\item
        The restriction $\alpha|_{G_{1}}$ ranges in a closed subgroup $H_{1}<H$,
        and is not cohomologous to a cocycle ranging 
        in any proper subgroup of $H$ of the form $h^{-1}H_{1} h \cap H_{1}$ 
        for some $h\in H$;
\item
        The action of $H_{1}$ on $H/H_{1}$ is tame;
\item
        $G_{1}\acts (X,\mu)$ is ergodic.
\end{itemize}
Then $\alpha$ ranges in the normaliser $N_{H}(H_{1})$ of $H_1$ and, passing to the
quotient, the cocycle $G\times X\to N_{H}(H_{1})/H_{1}$ is a homomorphism
$G\to N_{H}(H_{1})/H_{1}$ factoring through $G_{2}$.
\end{prop}

\begin{proof}
The arguments given on page~413 of~\cite{MonodShalom:CO:04}
apply word for word (with $G_1$ and $G_2=G'_1$ exchanged; 
the more specific assumptions in~\cite{MonodShalom:CO:04}
are not used for that proposition).
\end{proof}

The following is probably well-known.

\begin{lem}
\label{L:hull}
Let $G$ be a lcsc group, $X$ an ergodic probability $G$-space
and $\alpha:G\times X\to H$
a cocycle ranging in
an algebraic subgroup $H<L$ of an algebraic group $L$.

Then  the algebraic hull of $\alpha$ viewed as a cocycle in $H$ is the same as in $L$.
\end{lem}
\begin{proof}
Without loss of generality, $\alpha$ is Zariski-dense in
$H$. Let $H'<L$ be its hull in $L$; we have to show that
$H,H'$ are conjugated and it suffices to show that $H$ can
be conjugated into $H'$.

Let $\fhi:X\to L$ be a measurable map with $\alpha^{\fhi}$ ranging
in $H'$. 
The condition $\alpha^{\fhi}(g,x)\in H'$ can be written as
\[
        \fhi(g x) H'=\alpha(g,x)\fhi(x) H'.
\]
Since the left $H$-action on $L/H'$ is tame by Theorem~\ref{thm:tame:alg},
it follows by the ergodicity
that almost all $\fhi(x) H'$ belong to a single $H$-orbit $H\ell_0 H'$.
Hence for some measurable $\psi:X\to H$ we have
\[
        \fhi(x) H'=\psi(x) \ell_{0}H',
\]
which means that $\alpha^{\psi}$ ranges in $\ell_{0}H'\ell_{0}^{-1}$,
and thus in $\ell_{0}H'\ell_{0}\cap H$. By Zariski-density of
$\alpha$ in $H$, that range is $H$ and thus indeed
$\ell_0^{-1} H \ell_0 \se H'$.
\end{proof}

We shall prove the following statement, which is a variant of the ideas used by
Zimmer in~\cite{Zimmer:IHES:84} for the case where $L$ is compact.

\begin{prop}[Controlled Conjugation]
\label{P:taming}
Let $G,H$ be lcsc groups, $L<H$ a closed subgroup, $(X,\mu)$ a probability $G$-space
and $\alpha:G\times X \to H$ a cocycle. Fix some measurable norm on $H$ and 
assume that for some measurable $f:X\to H$ the cocycle $\alpha^{f}$
 ranges in $L$. Assume:
\begin{itemize}
\item[{(SG)}]
        The $G$-action on $X$ has a spectral gap;
\item[($L^{\infty}$)]
        $\|\alpha(g,-)\|\in L^{\infty}(X)$ for every $g\in G$.
\end{itemize}
Then there exists a measurable map $F:X\to H$ with $\|F(-)\|\in L^\delta(X)$
for some $\delta>0$ and such that $\alpha^{F}$ also ranges in $L$. In particular,
\begin{equation}
\tag{$L^{1}$}
        \int_X\log\|\alpha^F(g,x)\|\d\mu(x)\ < \infty\qquad(\forall\,g\in G).
\end{equation}
\end{prop}
\begin{lem}[Zimmer~\cite{Zimmer:IHES:84}]
\label{L:growth}
Let $G$ be a lcsc group and $(X,\mu)$ an ergodic probability $G$-space
with spectral gap. Then there is a compact subset $K\se G$ and $0<\lambda<1$
such that for every measurable $A_{0}\subset X$ 
with $\mu(A_{0})\ge 1/2$, there exists a sequence $\{g_n\}$ in $K$
such that the sequence of sets defined by $A_{n+1}=A_{n}\cup g_{n}A_{n}$
satisfies $\mu(A_{n})\ge 1-\lambda^{n}$.
\end{lem}
\begin{proof}
Recall from~\eqref{e:gap} that there is $K$ compact and $\epsilon>0$ such that
$$\forall\, \fhi \in L^{2}_{0}(X)\ \exists\, g\in K:\qquad \|\fhi\circ g- \fhi\|\ge \epsilon\cdot \|\fhi\|.$$
Define $\lambda=1-\epsilon^{2}/4$.
For a measurable set $A\subset X$, the function $p_{A}(x)=1_{A}(x)-\mu(A)$
is in $L^{2}_{0}(X)$ and has $\|p_{A}\|^{2}=\mu(A)\cdot (1-\mu(A))$.
If $A,B\subset X$ are of the same measure, then
$\|p_{A}-p_{B}\|^{2}=\mu(A\triangle B)$.
It now follows that for any $A\subset X$ with 
$\mu(A)>1/2$ there exists $g\in K$ such that 
\[
        \mu(gA\cap (X\setminus A))=\frac{1}{2}\mu(gA\triangle A)
        \ge \frac{\epsilon^{2}}{2}\mu(A)(1-\mu(A))\ge \frac{\epsilon^{2}}{4}\mu(X\setminus A)
\]
and therefore $1-\mu(A\cup gA)\ge \lambda (1-\mu(A))$.
Applying this argument inductively starting from $A=A_{0}$, the claimed estimates follow.
\end{proof} 
\begin{proof}[Proof of Proposition~\ref{P:taming}]
As  $\alpha^{f}$ ranges in $L$, for $g\in G$ a.e. $f(gx) L=\alpha(g,x)f(x) L$.
We shall construct a measurable $F:X\to H$ with $\|F\|^{\delta}$
integrable and such that a.e. $F(x) L=f(x) L$. Let $K$, $\epsilon$ and $\lambda$
be as before; we claim that the expression
$$C=\sup_{g\in K}\ \esssup_{x\in X}\, \|\alpha(g,x)\|$$
is finite. Indeed, the map $G\to L^\infty(X)$ defined by $g\mapsto \|\alpha(g,-)\|$ is weak-* continuous, and thus the image of $K$
is weak-*bounded. By the Banach--Steinhaus principle, it is norm bounded, whence the claim.

Choose now $\delta>0$ small enough to ensure $\lambda C^{\delta}<1$ and
$M<\infty$ large enough so that the set
$$A_{0}=\{x\in X : \|f(x)\|\le M\}$$
satisfies $\mu(A_{0})\geq 1/2$. Construct a sequence $\{A_{n}\}$
as in Lemma~\ref{L:growth}. We define $F$ on the conull set
$\bigcup_{n=0}^{\infty} A_{n}\subset X$ by
induction: Let $F(x)=f(x)$ for $x\in A_{0}$, and for $x\in A_{n+1}\setminus A_{n}$
$$F(x)= \alpha(g_{n},y) F(y)$$
where $x=g_{n} y$ with $y\in A_{n}$. Thus $\|F(x)\|\le C \|F(y)\|$ because $g_{n}\in K$.
This gives the estimate
\[
        \|F(x)\|\le M\cdot 1_{A_{0}}(x)+\sum_{n=1}^{\infty} M\cdot C^{n+1}\cdot
        1_{A_{n+1}\setminus A_{n}}(x).
\]
Since $\mu(A_{n+1}\setminus A_{n})\le \lambda^{n}$,
the choice of $\delta$ yields integrability of $\|F\|^{\delta}$.
We claim that $F(x) L=f(x) L$ holds on the conull set
$\bigcup_{n=0}^{\infty} A_{n}$.
Indeed, for $x\in A_{0}$ one has $F(x)=f(x)$; for $x\in A_{n+1}\setminus A_{n}$,
writing $y=g_{n}^{-1} x\in A_{n}$, we have
\[
        F(x) L=\alpha(g_{n},y)F(y) L=\alpha(g_{n},y)f(y) L=f(x) L
\]
using $F(y) L=f(y) L$ in the induction assumption.
\end{proof}
Let $\alpha:G\times X\to H$ be any cocycle satisfying the~$(L^1)$ condition, where
$G,H$ are lcsc groups and $H$ has some measurable norm. Consider the (finite) expression
\begin{equation}\label{e:coc-length}
\ell(g)=\int_{X} \log\|\alpha(g,x)\|\d\mu(x).
\end{equation}
We observe that $\ell$ is subadditive since
\begin{eqnarray*}
        \ell(g_{1}g_{2}) &\le& 
        \int_{X}\log\|\alpha(g_{1},g_{2} x)\|\d\mu(x) +\int_{X}\log\|\alpha(g_{2},x)\|\d\mu(x)\\
        &=&\ell(g_{1})+\ell(g_{2}).
\end{eqnarray*}
In other words, $\exp(\ell)$ is a measurable norm on $G$; moreover,
$\ell$ is bounded on compact sets (Lemma~\ref{L:bdd}).

\begin{rem}\label{R:word-length}
Subadditivity implies also that if $G$ is a finitely generated group with some word-length $\ell_G$,
then $\ell$ admits a linear bound in terms of $\ell_G$.
\end{rem}
\section{Cocycles, Unitary Representations and Invariant Measures}
\label{S:unitary-measures}%
This section contains some general considerations relating quasi-regular
representations and existence of invariant measures. These are used
in the proof of Theorem~\ref{T:coc-Furstenberg}, but seem to be of independent interest.

\medskip

Let $G$ be lcsc group and $(X,\mu)$ an ergodic probability $G$-space.
Denote by $\pi$ the unitary representation on $L^{2}(X)$
and by $\pi_{0}$ its restriction to $L^2_0(X)$.
Let $B$ be a compact metrisable space with some given Borel-regular
probability measure $\nu$ of full support, 
and let  $H<\Homeo(B)$ be a (lcsc) group of homeomorphisms which leaves the 
measure class of $\nu$ invariant. 
Let $\alpha:G\times X\to H$ be a measurable cocycle.
We denote by $\tilde\pi$ the quasi-regular unitary $G$-representation on 
$L^{2}(X\ltimes B,\mu\times\nu)$. Therefore, writing everything explicitly,
$$(\tilde\pi(g^{-1})F)(x,b)\ =\ \Big(\frac{\d\alpha(g,x)^{-1}_{*}\nu}{\d\nu}(b)\Big)^{1/2}
   F(g x, \alpha(g,x)b).$$
Let $\Prob_{\mu}(X\times B)$ denote the space of all probability measures on 
$X\times B$ that project to $\mu$. The $G$-action on $X\ltimes B$ defines
a $G$-action on $\Prob_{\mu}(X\times B)$ since $\mu$ is $G$-invariant.
By disintegration with respect to $\mu$, any measure $\eta\in \Prob_{\mu}(X\times B)$
can be written as
\[
        \eta=\int_{X}(\delta_{x}\times \eta_{x})\d\mu(x)
\]
where $x\mapsto \eta_{x}\in\Prob(B)$ is a measurable map. Such a measure $\eta$
is $G$-invariant if and only if $\eta_{g x}=\alpha(g,x)_{*}\eta_{x}$ holds for all
$g\in G$ and $\mu$-a.e. $x\in X$.

It is straightforward to verify that $\tilde{\pi}$ contains $\one_G$
if and only if $G$ preserves a measure in $\Prob_{\mu}(X\times B)$
that is absolutely continuous with respect to $\mu\times\nu$. We shall
however need a more refined statement relating the existence of arbitrary
$G$-fixed measures in $\Prob_{\mu}(X\times B)$
and the spectral properties of $\tilde{\pi}$, as follows.
\begin{prop} 
\label{P:inv-in-Pm}
If $\pi_{0}\not\succ\one_{G}$ but $\tilde{\pi}\succ \one_{G}$, then
$G$ preserves some measure $\eta\in \Prob_{\mu}(X\times B)$. 
\end{prop}
\begin{proof}
Suppose that $F_{n}$ is a sequence of unit vectors in $L^{2}(X\times B)$ with 
\[
        d_{n}(g)=\|\tilde\pi(g)F_{n}-F_{n}\| \lra 0
\] 
uniformly on compact sets.
Upon replacing $F_{n}$ by $|F_{n}|$ the value $d_{n}$ will only decrease,
so we can assume $F_{n}\ge 0$.
Consider the sequence of unit vectors $\{f_{n}\}\in L^{2}(X)$ given by
\[
        f_{n}(x)=\left(\int_{B}F_{n}(x,b)^{2}\d\nu(b)\right)^{1/2}.
\]
We claim that $f_{n}$ is a sequence of almost invariant unit vectors for $\pi$,
and thus the assumption  $\pi_{0}\not\succ \one_{G}$ gives
\begin{equation}
\label{e:fn-to-1}
        \| f_{n}-{\bf 1}\|\lra 0.
\end{equation}
To verify almost invariance of $\{f_{n}\}$, note 
that in view of the elementary inequality $|a-b|^{2}\le |a^{2}-b^{2}|$ 
we have
\begin{align*}
        \left\| f_{n}-\pi(g)f_{n} \right\|^{2} &\le \int_{X} \left|f_{n}^{2}-(\pi(g)f_{n})^{2}\right|\d\mu\\
        &\le \int_{X}\int_{B} \left| F_{n}^{2}-(\tilde\pi(g)F_{n})^{2}\right|\d\mu\d\nu \\
        &=\int_{X}\int_{B} |F_{n}+\tilde\pi(g)F_{n}|\cdot\left|F_{n}-\tilde\pi(g)F_{n}\right|\d\mu\d\nu\\
        &\le \left\|F_{n}+\tilde\pi(g) F_{n}\right\|\cdot\left\|F_{n}-\tilde\pi(g) F_{n}\right\|\\
        &\le 2\cdot d_{n}(g)\ \longrightarrow 0
\end{align*}
uniformly on compact sets in $G$; hence~\eqref{e:fn-to-1} follows.
We shall now define a probability measure $\eta$ on $X\times B$ 
as a functional on $L^{1}(X,\mu)\otimes C(B)$.
Let $\psi_{i}:X\to[0,1]$ be a sequence of measurable functions spanning
a dense subspace in $L^{1}(X,\mu)$, and $\fhi_{j}:B\to [0,1]$
be a sequence of continuous functions spanning a dense subspace in $C(B)$.
Assume that $\psi_{0}={\bf 1}_{X}$ and $\fhi_{0}={\bf 1}_{B}$ constant one functions.
For each $i,j$ the following sequence in $n$ is non-negative and satisfies
\begin{eqnarray*}
        \ip{\psi_{i}\otimes\fhi_{j}}{F_{n}^{2}}
        &=&\int_{X}\int_{B} \psi_{i}(x)\fhi_{j}(b)F_{n}(x,b)^{2}\d\nu(b)\d\mu(x)\\
        &\le& \int_{X} f_{n}^{2}\d\mu
        \quad \lra \quad 1.
\end{eqnarray*}
Applying the diagonal process, one finds a subsequence $\{n_{k}\}$ along which 
the LHS above converges for all $i,j$. 
We can now define $\eta$ by
\[
        \int_{X\times B}\psi_{i}(x)\fhi_{j}(b)\d\eta(x,b)
        =\lim_{k\to\infty}\ip{\psi_{i}\otimes\fhi_{j}}{F_{n_{k}}^{2}}.
\]
More precisely, extending $\eta$ linearly to the span of $\psi_{i}\otimes\fhi_{j}$
we note that it is a positive, normalised functional satisfying
\[
        \ip{\psi_{i}\otimes {\bf 1}_{B}}{\eta}
        =\lim_{k\to\infty} \int_{X}\psi_{i} f_{n}^{2}\d\mu
        =\int_{X}\psi_{i}\d\mu
\]
Thus it corresponds to a measure $\eta$ on $X\times B$ projecting onto $\mu$.
This measure is $G$-invariant; indeed for fixed $\psi_{i}$, $\fhi_{j}$ we have
\[
        \ip{\psi_{i}\otimes\fhi_{j}}{g\eta-\eta}
        =\lim_{n\to\infty}
        \ip{\psi_{i}\otimes\fhi_{j}}{(\tilde\pi(g) F_{n_{k}})^{2}-F_{n_{k}}^{2}}
\]
whilst
\begin{multline*}
\left|\ip{\psi_{i}\otimes\fhi_{j}}{(\tilde\pi(g) F_{n_{k}})^{2}-F_{n_{k}}^{2}}\right|\\
\le \int_{X}\int_{B}\psi_{i}\fhi_{j}
\left|\tilde{\pi}(g)F_{n_{k}}+F_{n_{k}}\right|
\cdot\left|\tilde{\pi}(g)F_{n_{k}}-F_{n_{k}}\right|
\d\mu\d\nu\\
\le \left\|\tilde{\pi}(g)F_{n_{k}}+F_{n_{k}} \right\|
\cdot\left\|\tilde{\pi}(g)F_{n_{k}}-F_{n_{k}} \right\|
\le 2d_{n_{k}}\lra 0.
\end{multline*}
\end{proof}

We now specialise to the case where: $H=\SL_{d}(\RR)$, $B=\PP^{d-1}$ is
the projective space and $\nu$ is the unique $\SO_{d}(\RR)$-invariant probability
measure on $B$.
\begin{lem}[Zimmer's Cocycle Reduction]
\label{L:no-inv-in-Pm}
Let $G$ be a lcsc group, $(X,\mu)$ an ergodic probability $G$-space and
$\alpha:G\times X\to \SL_{d}(\RR)$ a cocycle.
If the corresponding $G$-action on $X\ltimes \PP^{d-1}$
preserves a probability measure $\eta$ projecting to $\mu$, then
either
\begin{enumerate}
\item
        $\alpha$ is cohomologous to a cocycle $\alpha^{\prime}:G\times X\to \SO_d(\RR)$, or 
\item
        $\alpha$ is cohomologous to a cocycle $\alpha^{\prime}:G\times X\to L$
        where $L$ has a finite index subgroup that is reducible on $\RR^{d}$.
\end{enumerate}
If $\eta\prec \mu\times\nu$, then case~{\rm (1)} holds.
\end{lem}
\begin{proof}
A $G$-invariant measure has the form $\eta=\int_{X} (\delta_{x}\times\eta_{x})\d\mu(x)$
with 
\[
        \eta_{g x}=\alpha(g,x)_{*}\eta_{x}\qquad \mu\text{-a.e. on } X.
\]
The action of $H=\SL_{d}(\RR)$ on the space $\Prob(\PP^{d-1})$ of probability
measures is tame by Theorem~\ref{thm:tame:alg}.
In view of the ergodicity of $G$-action on $(X,\mu)$ this implies that
$\mu$-almost all $\eta_{x}$ lie on a single $H$-orbit: 
\[
        \eta_{x}=\fhi(x)\eta_{0}
\] 
for some $\eta_{0}\in \Prob(\PP^{d-1})$ and a measurable map $\fhi:X\to H$.
Denoting by $H_{0}=\{h\in \SL_{d}(\RR)\ :\  h_{*}\eta_{0}=\eta_{0}\}$ the stabiliser
of this measure, we get that the cocycle
\[
        \alpha^{\prime}(g,x)=\fhi(g x)^{-1}\,\alpha(g,x)\,\fhi(x)
\]
ranges in $H_{0}$.
Furstenberg's Lemma, which can be found \emph{e.g.} as Corollary~3.2.2 in~\cite[Cor 3.2.2]{Zimmer:book:84},
implies that $H_{0}$ is either compact or virtually reducible on $\RR^{d}$, and that moreover the former
case holds when $\eta_{0}\prec \nu$. It remains only to observe that $\eta\prec \mu\times\nu$ implies $\eta_{0}\prec \nu$
and to recall that any compact subgroup of $\SL_{d}(\RR)$ can be conjugated into $\SO_d(\RR)$.
\end{proof}

\section{Random Walks and a Furstenberg Condition}
\label{S:RW}
This section investigates the growth of matrix-valued cocycles
along random walks. The main result is a cocycle analogue
of the famous Furstenberg condition for positivity
of the top Lyapunov exponent. We recall the definition
of the following integrability condition:
\begin{equation}
\tag{$L^{1}$}
        \int_X\log\|\alpha(g,x)\|\d\mu(x)\ <\ \infty\qquad(\forall\,g\in G).
\end{equation}
Our presentation is based on~\cite{Furman:RT:02}.

\begin{thm}[Cocycle Version of Furstenberg's Theorem]
\label{T:coc-Furstenberg}
Let $G$ be a lcsc group, $(X,\mu)$ an ergodic probability $G$-space with spectral gap
and $\alpha:G\times X\to \SL_{d}(\RR)$ a cocycle satisfying~$(L^1)$. Suppose that $\alpha$
is not equivalent to a cocycle ranging into a compact or virtually
reducible subgroup of $\SL_{d}(\RR)$.

Then, for any absolutely continuous generating measure $\sigma$ on $G$:
\[
        \liminf_{n\to\infty} \int_{G}\int_{X} \frac{1}{n}\log\|\alpha(g,x)\|\d\mu(x)\d\sigma^{*n}(g)>0.
\]
\end{thm}
\noindent
This result will be applied when $\sigma$ satisfies
\[
        \int_{G}\int_{X} \log\|\alpha(g,x)\|\d\mu(x)\d\sigma(g)<\infty,
\]
in which case the limit in the theorem is finite and will be denoted by $\lambda_{1}=\lambda_{1}(\alpha,\sigma)$.
We begin with some preparations and recall the subadditive function $\ell$ from~\eqref{e:coc-length}.

\begin{lem}
\label{L:goodmeas}
There exist symmetric absolutely continuous probability measures $\sigma$ of full support on $G$
such that $\ell\in L^{1}(G,\sigma)$, i.e.:
\begin{equation}\label{e:mu-moment}
        \int_{G} \int_X \log\|\alpha(g,x)\|\d\mu(x)\d\sigma(g) <\infty.
\end{equation}
\end{lem}
\noindent
(In the case where $G$ is discrete, this is obvious.)
\begin{proof}
Let $U$ be a compact neighbourhood of the identity in $G$; in particular,
$m_{G}(gU)=m_{G}(U)<\infty$ for all $g$. By Lemma~\ref{L:bdd}, $\ell$ is bounded on $U$.
Choose a countable set $\{g_{n}\}$ so that $G=\bigcup g_{i}U$ and let 
$\sigma'=\sum 2^{-n}\ell(g_{i})^{-1}\cdot m_{g_{n}U}$ where $m_{A}$
is the restriction of $m_{G}$ to a Borel subset $A\subset G$.
Then $\sigma'$ is a finite positive measure, equivalent to $m_{G}$,
and $\ell\in L^{1}(G,\sigma')$. We may now take $\sigma$ to be the 
normalised symmetrised measure
\[
        \sigma(E)= (\sigma'(E)+\sigma'(E^{-1})) \big/ 2\sigma'(G) \qquad (E\subset G).
\]
\end{proof}

\begin{proof}[Proof of Theorem~\ref{T:coc-Furstenberg}]
We consider the $G$-space $X\ltimes \PP^{d-1}$ as in Lemma~\ref{L:no-inv-in-Pm}
and recall that we chose for $\nu$ the (unique) $\SO_d(\RR)$-invariant probability measure on
$\PP^{d-1}$.
Consider the quasi-regular $G$-representation $\tilde{\pi}$
on $L^{2}(X\ltimes\PP^{d-1},\mu\times\nu)$ defined in Section~\ref{S:unitary-measures}.
Applying Proposition~\ref{P:inv-in-Pm} and Lemma~\ref{L:no-inv-in-Pm},
we deduce that $\tilde\pi$ has a spectral gap. By the Kesten-type characterization
(Lemma~\ref{l:kesten}), it follows
\begin{equation}\label{e:SG:pi}
\|\tilde\pi(\sigma)\|_\mathrm{sp}<1
\end{equation}
for an arbitrary absolutely continuous generating probability measure $\sigma$ on $G$.
This gap will allow us to estimate the growth of the cocycle using the following lemma;
for shorter notation, we denote by
\[
        \ro(h,\xi)=\frac{\d h^{-1}_*\nu}{\d\nu}(\xi)        
\]
the Radon--Nikod\'ym derivative for $h\in \SL_{d}(\RR)$, $\xi\in\PP^{d-1}$.

\begin{lem}
\[
        \|h\|\ge \left(\int_{\PP^{d-1}}\sqrt{\ro(h,\xi)}\d\nu(\xi)\right)^{-d/2}
\]
\end{lem}
\begin{proof}
Let  $\ro_{max}(h)=\max_{\xi\in\PP^{d-1}}\ro(h,\xi)$.
Since $\int_{\PP^{d-1}}\ro(h,\xi)\d\nu(\xi)=1$, we have
\[
         1\le   \sqrt{\ro_{max}(h)}\cdot \int_{\PP^{d-1}} \sqrt{\ro(h,\xi)}\d\nu(\xi).
\]
It now suffices to show that $\ro_{max}(h)=\|h\|^{d}$.
Using the Cartan (polar) decomposition and since $\nu$ is $\SO_d(\RR)$-invariant,
it is enough to consider for $h$ a diagonal matrix $h={\rm diag}[a_{1},\dots,a_{d}]$
with $a_{1}\ge a_{2}\ge\dots\ge a_{d}>0$ and $a_{1}\cdots a_{d}=1$, where
\[
        \ro_{max}(h)=\ro(h,\RR e_{1})
        =\frac{a_{1}}{a_{2}}\times\cdots\times \frac{a_{1}}{a_{d}}=a_{1}^{d}.
\] 
On the other hand, $\|h\|=a_{1}$.
\end{proof}

Using the above estimate we deduce:
\begin{eqnarray*}
        \int_{G}\int_{X} &&\log\|\alpha(g,x)\|\d\mu(x)\d\sigma(g)\\
        &&\ge  \int_{G}\int_{X} - \frac{d}{2}\cdot \log
                \left(\int_{\PP^{d-1}}  \sqrt{\ro(\alpha(g,x),\xi)}\d\nu(\xi)\right)
                \d\mu(x)\d\sigma(g)\\
        &&\ge - \frac{d}{2}\cdot \log\left(
                \int_{G}\int_{X\times\PP^{d-1}}\sqrt{\ro(\alpha(g,x),\xi)}\,
                d\sigma(g)\d(\mu\times\nu)(x,\xi)\right)\\
        &&=  -\frac{d}{2}\cdot\log\ip{\tilde\pi(\sigma){\bf 1}}{{\bf 1}}
        \ge -\frac{d}{2}\cdot \log\|\tilde\pi(\sigma)\|.
\end{eqnarray*}
Replacing $\sigma$ by $\sigma^{*n}$, where $\sigma$ is generating, we get
\begin{multline*}
        \liminf_{n\to\infty}\ \frac{1}{n}\int_{G}\int_{X}\log\|\alpha(g,x)\|\d\mu(x)\d\sigma^{*n}(g)\\
        \ge \lim_{n\to\infty} \ -\frac{d}{2}\cdot \log\|\tilde\pi(\sigma^{*n})\|^{1/n}
         = - \frac{d}{2}\cdot\log\|\tilde\pi(\sigma)\|_\mathrm{sp}.
\end{multline*}
This last term is strictly positive by~\eqref{e:SG:pi}, concluding the proof of Theorem~\ref{T:coc-Furstenberg}.
\end{proof}

We shall now recall Oseledets' multiplicative ergodic theorem
and establish some additional information that becomes
available when $\lambda_1>0$ thanks to the above theorem.

\smallskip

Let $\sigma$ be a probability measure as in Lemma~\ref{L:goodmeas}. 
Consider the one-sided Bernoulli shift $\teta$ acting on 
$\Omega=G^{\NN}$ equipped with the product measure $\sigma^{\NN}$ by $(\teta \omega)_{i}=\omega_{i+1}$.
Using the $G$-action on $X$, one defines a transformation 
$T$ on $Z=\Omega\times X$, preserving the measure $\sigma^{\NN}\times \mu$, by 
\[
        T(\omega,x)=(\teta\omega, \omega_{1} x).
\]
In fact, by Kakutani's random ergodic theorem~\cite{Kakutani51}, the assumption that $\sigma$ has full support 
on $G$ together with ergodicity of the $G$-action on $X$
implies that $T$ is ergodic.
The cocycle $\alpha:G\times X\to\SL_{d}(\RR)$ gives rise to a 
function $A:Z\to \SL_{d}(\RR)$
\[
        A(\omega,x)=\alpha(\omega_{1},x)
\] 
for which $\log \|A(-)\|$ is in $L^{1}(Z,\sigma^{\NN}\times \mu)$ 
by~\eqref{e:mu-moment}.
The associated $\NN$-cocycle $\NN\times Z\to \SL_{d}(\RR)$ takes the following form
\[
        A_{n}(\omega,x)=(A\circ T^{n-1})\cdots (A\circ T) A (\omega,x)
        =\alpha(\omega_{n}\cdots \omega_{1}, x).
\]
To such a function one associates the non-negative quantity
\begin{align*}
        \lambda_{1}= \lambda_{1}(\alpha, \sigma) &=
                \liminf_{n\to\infty}  \int_{X}  \frac{1}{n}\log\|A_{n}(z)\|\d z\\
                &=\liminf_{n\to\infty}  \int_{X} \int_{G} \frac{1}{n}\log\|\alpha(g,x)\|\d\mu(x)\d\sigma^{*n}(g).
\end{align*}
It follows from Kingman's subadditive ergodic theorem that the above $\liminf$ 
is actually a limit (converging to the infimum); moreover,
the convergence to the constant function $\lambda_{1}$
holds not only for the integral, but also almost-everywhere and in $L^{1}$.
 
\medskip 
 
When $\lambda_{1}>0$, Oseledets' theorem (\cite{Oseledets}; see also~\cite{Raghunathan79,Ruelle79}) gives
further structure, namely there exist:
\begin{itemize}
\item
        An integer $1<k\le d$, integers $d=d_{1}>\dots>d_{k}>d_{k+1}=0$ and reals
        $\lambda_{1}>\dots>\lambda_{k}$;
\item
        A measurable family $\{ E_{\omega, j}(x) \}$ of $(d_{1}, d_{2},\dots,d_{k})$-flags
        \[
                \RR^{d}=E_{\omega,1}(x)\supset E_{\omega,2}(x)\supset\dots\supset E_{\omega,k}(x)
        \]
        with $\dim E_{\omega,j}=d_{j}$ and such that for a.e. $(\omega,x)\in \Omega\times X$  
        \[
        E_{\omega,j}(x)=\left\{ v\in \RR^{d} \ :\ 
                \limsup_{n\to\infty} \frac{1}{n}\log\|\alpha(\omega_{n}\cdots\omega_{1},x)v\|
                \le \lambda_{j}\right\}.
        \]
\end{itemize}

\begin{prop}
\label{P:inv-of-Lyap}
Let $G=G_1\times G'_1$ be a lcsc group, $(X,\mu)$ a probability $G$-space
on which $G_1$ is ergodic and $\alpha:G\times X\to\SL_{d}(\RR)$ a cocycle
satisfying~$(L^1)$.

Then $G'_1$ leaves invariant the characteristic filtrations associated
to random walks on $G_1$. More precisely, if $\sigma$ is a probability measure
on $G_1$ satisfying~\eqref{e:mu-moment} and   
$\lambda_{1}(\alpha|_{G_1}, \sigma)>0$, then
\[
        \alpha(h,x) E_{\omega,j}(x)= E_{\omega,j}(h x)\qquad (\forall\,h\in G'_1, \text{a.e. }x\in X, \omega\in\Omega).
\]
\end{prop}
\begin{proof}
Fix some $h\in G'_1$. Given $g\in G_1$, $x\in X$ and $v\in \RR^{d}$,
we write $y=h x$ and $w=\alpha(h,x)^{-1}v$. Since $g$ and $h$ commute, we have
\begin{align*}
        \alpha(g,x) w&=\alpha(h^{-1}g\, h, x)w\\
        &=\alpha(h^{-1}g, h x)\,\alpha(h,x) w=\alpha(h^{-1}g, y) v\\
        &=\alpha(h^{-1}, g y)\,\alpha(g, y) v
\end{align*}
which gives the estimate
\[
        \Big| \log\|\alpha(g,x) w\|-\log\|\alpha(g,y) v\|\Big|\le f(g y)
\]
with $f(-)=\log\|\alpha(h^{-1},-)\| \in L^{1}(X)$.
Recall that for any $L^{1}$-function $\fhi$ on an ergodic
system $(Z,T)$, Birkhoff's ergodic theorem implies
\[
        \frac{1}{n}\fhi(T^{n}z)=\frac{1}{n}\fhi(z)+\frac{1}{n}\sum_{k=0}^{n-1}
        \left(\fhi\circ T-\fhi\right)(z)\ \longrightarrow 0
\]
for a.e. $z$. Viewing $f$ as an $L^{1}$-function on $Z=\Omega\times X$, we deduce
\[
        \left| \frac{1}{n}\log \|\alpha(g_{n},x) w\|-\frac{1}{n}\log \|\alpha(g_{n},h x)v\|\right|
        \le \frac{1}{n} f(g_{n} y)\longrightarrow 0
\]
for a.e. $x\in X$ and a.e. $\omega\in\Omega$. This shows that $v\in E_{\omega,j}(h x)$ is
equivalent to $w\in E_{\omega,j}(x)$, finishing the proof.
\end{proof}

\section{The Cocycle Superrigidity Theorems}
\subsection{Semi-Simple Hulls}\label{S:s-s:hull}
We will construct the homomorphism $\ro: G\to H$ as a product of homomorphisms
$\ro_i: G_i\to H_i$, where $H=H_1\times\cdots\times H_n$ will be a splitting of the given
connected centre-free group $H$ into (possibly trivial) connected semi-simple Lie groups.

We first consider the case of $H$ \emph{simple} (thus the above splitting will have only one
non-trivial factor). We recall that it was assumed in 
Theorem~\ref{T:cocycle1} that $H$ has no compact factors; in particular, $H$ 
is non-compact (compare with Section~\ref{S:compact}).  

\begin{proof}[Proof of Theorem~\ref{T:cocycle1} for $H$ simple]
The assertion to prove is that there exists a single factor $G_{i_{1}}$ and
a representation $\ro_{i_{1}}:G_{i_{1}}\to H$ with Zariski-dense image
such that
\[
        \alpha(g,x)=f(g x)^{-1}\,\ro_{i_{1}}\,(g_{i_{1}})\,f(x)
\]
for some $f:X\to H$.

\medskip

Write $G=G_{1}\times G_{1}'$ where $G_{1}'=G_{2}\times\cdots\times G_{n}$, and 
let $H_{1}$ and $H_{1}'$ be the algebraic hulls in $H$ of the restrictions
$\alpha|_{G_1}$ and $\alpha|_{G_{1}'}$, respectively.
Thus there are measurable maps $f,f':X\to H$ such that
$\alpha^{f}$ ranges in $H_{1}$ on $G_{1}\times X$
and $\alpha^{f'}$ ranges in $H_{1}'$ on $G_{1}'\times X$.

Since $H_{1}<H$ is an inclusion of algebraic groups,
the action $H_{1}\acts H/H_{1}$ is tame by Theorem~\ref{thm:tame:alg}.
Hence, applying Proposition~\ref{P:CRL}, we deduce that the cocycle 
$\alpha^f$ ranges in $N_{H}(H_{1})$. Since $\alpha$ was assumed to be Zariski-dense in $H$,
the same holds for $\alpha^f$. But $N_{H}(H_{1})$ being an algebraic subgroup of $H$, 
it follows that $N_{H}(H_{1})=H$. As $H$ is simple, we have either $H_{1}=H$ or $H_{1}=\{e\}$.

\smallskip

\emph{Case $H_{1}=\{e\}$}. It follows from Proposition~\ref{P:CRL} that
the restricted cocycle $\alpha^{f}|_{G'_1}$ is a homomorphism $\ro':G_{1}'\to H$.
Note that the image $\ro'(G_{1}')$ is Zariski-dense in $H$.
Observe that for each $2\le i\le n$ the image $\ro'(G_{i})$ is normalised by
$\ro'(G_{1}')$, and therefore by all of $H$, hence each $\ro'(G_{i})$ is
either trivial, or is Zariski-dense in $H$. 
However, for all but one $2\le i\le n$ the image $\ro'(G_{i})$ is trivial,
for otherwise the simple group $H$ would contain two 
commuting Zariski-dense subgroups $\ro'(G_{i})$ and $\ro'(G_{j})$, which is impossible.
Hence $\ro'$ factors through a Zariski-dense homomorphism $\ro_{i_{1}}:G_{i_{1}}\to H$
of a single factor.
 
\smallskip

\emph{We can thus assume $H_{1}=H$ for the remainder of this proof}.
Applying Proposition~\ref{P:CRL} to $\alpha^{f'}|_{G'_1}$,
we deduce that $H_{1}'\normal H$ and therefore either $H_{1}'=\{e\}$ or $H_{1}'=H$.
If $H_{1}'=\{e\}$, then Proposition~\ref{P:CRL} shows that the cocycle 
$\alpha^{f'}$, being trivial on $G_{1}'\times X$, is a homomorphism $\ro_{1}:G_{1}\to H$
when restricted to $G_{1}\times X$. Therefore,
the main point is to prove the following key proposition, for which we shall also indicate an
alternative approach in Section~\ref{S:splitting}.
\begin{prop}
\label{P:H1H2H}
One cannot have $H_{1}'=H$.
\end{prop}

Towards a contradiction, let us assume $H_{1}'=H$.
Choose an irreducible faithful representation $\pi:H\to\SL_{r}(\RR)$ and consider
the resulting linear cocycle
\[
        \beta=\pi\circ\alpha:G\times X\xrightarrow{\,\alpha\,}H\xrightarrow{\,\pi\,}\SL_{r}(\RR).
\]

\begin{lem}
\label{L:noncpt+strirr}
The restrictions $\beta|_{G_1}$ and $\beta|_{G'_1}$ are both not
cohomologous to cocycles ranging in a compact subgroup, nor in a subgroup leaving
invariant a finite family of proper subspaces in $\RR^{r}$.
\end{lem}
\begin{proof}[Prof of the lemma]
Recall that compact groups, or groups leaving invariant a finite union of 
subspaces in $\RR^{r}$, are both examples of $\RR$-algebraic groups in $\SL_{r}(\RR)$.
Lemma~\ref{L:hull} implies that the algebraic hull of $\beta|_{G_1}$ is $\pi(H)$.
Had $\beta|_{G_1}$ been cohomologous to a cocycle ranging in an algebraic
subgroup $L<\SL_{r}(\RR)$, it would follow that $\pi(H)<L$ 
up to replacing $L$ by its conjugate.
However $\pi(H)$ is not a subgroup of a compact group
and does not leave invariant a finite union of proper subspace,
because $\pi(H)$ is connected, irreducible and non-compact.

The same argument can be applied to $\beta|_{G'_1}$ since we are assuming $H_{1}'=H$.
\end{proof} 

The pointwise integrability condition ($L^{1}$) on $\alpha$ is inherited by $\beta$,
and passed on to $\beta|_{G_1}$. Choose a probability measure $\sigma$ on $G_{1}$
as in Lemma~\ref{L:goodmeas} for $\beta|_{G_1}$. That is, $\sigma$ is absolutely continuous,
symmetric and satisfies the average integrability condition
\[
        \int_{G_{1}}\int_{X}\log\|\beta|_{G_1}(g,x)\|\d\mu(x)\d\sigma(g)\ <\infty.
\]
One may now look at the behavior of $\beta|_{G_1}$ along a path of a $\sigma$-random walk
as is described in Section~\ref{S:RW}.
In particular, Lemma~\ref{L:noncpt+strirr} allows us to apply Theorem~\ref{T:coc-Furstenberg} 
which ensures strict positivity of the top Lyapunov exponent
\[
        \lambda(\beta|_{G_1},\sigma)=\lim_{n\to\infty} \frac{1}{n}\int_{G_{1}}\int_{X}
        \log\|\beta|_{G_1}(g,x)\|\d\mu(x)\d\sigma^{*n}(g)\ >0.
\]
This fact yields the non-triviality of
the Lyapunov filtrations (see Section~\ref{S:RW}). 
In particular, for some proper intermediate dimension $0<d_2<r$  there 
is a measurable family $\{E_{\omega,2}(x)\}_{\omega\in\Omega}$  
of $d_2$-dimensional vector subspaces $E_{\omega,2}(x)<\RR^{r}$ 
which are intrinsically defined by $\beta|_{G_1}:G_{1}\times X\to\SL_{r}(\RR)$
(describing the exceptional space of ``slow'' vectors under the $\omega$-path
of the random walk).

Being naturally associated to $G_{1}$-action each of these sections 
$E_{\omega,2}:X\to {\rm Gr}(r,d_2)$ is invariant under the action of the commuting
group $G_{1}'$. Namely by Proposition~\ref{P:inv-of-Lyap},
for almost every $\omega\in\Omega$ the section $E_{\omega,2}$ satisfies
\[ 
        \beta|_{G'_1}(g,x) E_{\omega,2}(x)=E_{\omega,2}(g x)\qquad(g\in G_{1}').
\]
Fix a $d_2$-dimensional subspace $E_{0}<\RR^{r}$, and choose a 
measurable map $\fhi:X\to \SO_{r}(\RR)$ such that $E_{\omega,2}(x)=\fhi(x)E_{0}$.
Then the $\fhi$-conjugate of $\beta|_{G'_1}:G_{1}'\times X\to \SL_{r}(\RR)$ 
ranges in the group $\{g\in\SL_{d}(\RR): g E_{0}=E_{0}\}$,
contradicting Lemma~\ref{L:noncpt+strirr}.
This contradiction proves Proposition~\ref{P:H1H2H} and therefore
completes the proof of Theorem~\ref{T:cocycle1} in the case of $H$ being simple.
\end{proof}

\begin{proof}[Completion of the proof of Theorem~\ref{T:cocycle1} in general]
Now consider a general connected centre-free semi-simple Lie group $H$ without compact factors.
Then $H$ can be written as a direct product of simple factors $H=S_{1}\times\cdots\times S_{N}$.
(Indeed, the connectedness implies that $H$ is the product of its simple factors, and the triviality of the centre
implies that the latter product is direct.)
Let $\pi_{t}:H\to S_{t}$ denote the projections and apply the case of a simple target to
the cocycles 
\[
        \pi_{t}\circ \alpha:G\times X\lra H\lra S_{t}\qquad(t=1,\dots,N).
\]
We get homomorphisms $\tau_{t}:G_{i(t)}\to S_{t}$, where $i(t)\in\{1,\dots,n\}$,
and corresponding conjugating maps $f_{t}:X\to S_{t}$. For $j=1, \ldots, n$ we define
$$H_j= \prod_{i(t)=j} S_{t}, \qquad \ro_{j}= \prod_{i(t)=j} \tau_{t}:\ G_j\to H_j$$
being understood that the product over an empty set is the trivial group or morphism.
We have an identification $H=H_1\times \cdots\times H_n$.
Arranging the maps $f_{t}$ accordingly into a single map $f:X\to H$, we have
\[
        \alpha(g,x)=f(g x)^{-1}\, (\ro_{1}(g_1),\ldots,\ro_{n}(g_n))\,f(x).
\]
Note that $\ro_{j}(G_j)$ is Zariski-dense in $H_{j}$ because 
$\alpha$ is Zariski-dense in $H$; define $\ro=\prod_i \ro_i$. This concludes the proof.
\end{proof}

\subsection{Unrestricted Hulls}\label{S:cocycle2}
We turn now to the proof of Theorem~\ref{T:cocycle2}, where the cocycle
is no longer assumed to be Zariski-dense in a connected semi-simple Lie group.
There is a number of issues to address before it is possible to reduce the argument
to the proof of the semi-simple case, notably:

\smallskip

--- Controlled conjugation into the algebraic hull;

--- Non-connectedness of the algebraic hull;

--- The amenable radical.

\medskip 

\begin{proof}[Proof of Theorem~\ref{T:cocycle2}]
Keep the notation of the theorem. Let
\[
        \beta(g,x)=f(g x)^{-1}\,\alpha(g,x)\,f(x)
\]
be a conjugate of $\alpha$ that ranges in $L$. Conditions (SG) and
($L^{\infty}$) allow us to apply Proposition~\ref{P:taming} in order to choose 
$f$ such that $\|f\|^\delta$ is integrable for some $\delta>0$.
Recall that $L^{0}\normal L$ is the connected component of the identity.
Consider the finite extension $Y=X\ltimes (L/L^0)$ of the $G$-action on $X$,
which is the skew product by $\beta$ endowed with the invariant measure $\nu=\mu\times m_{L/L^0}$.
It is shown in~\cite[9.2.6]{Zimmer:book:84} that this action is ergodic
(otherwise, one could conjugate $\beta$ to a cocycle ranging in an intermediate  
subgroup $L^{0}<L_{1}<L$, which is incompatible with Zariski-density).

The point of the finite extension $\pi:Y\to X$ is that
the lift of the cocycle $\beta$ to 
$G\times Y\to L$ becomes cohomologous to a cocycle ranging in $L^{0}$.
Indeed, if $\xi:L/L^{0}\to L$ is a choice of coset representatives and
$\fhi(x,\ell L^{0})=\xi(\ell L^{0})$, then the cocycle $G\times Y\to L$
defined by $\fhi(g y)^{-1}\,\beta(g,\pi(y))\,\fhi(y)$ ranges in $L^{0}$
(in fact, $L^0$ is its algebraic hull~\cite[9.2.6]{Zimmer:book:84}).

\medskip

We now have a finite ergodic $G$-equivariant extension 
$\pi:Y\to X$ such that the lift of $\alpha$
to $G\times Y\to \GL^1_{d}(\RR)$ is cohomologous
to a cocycle ranging in $L^{0}$:
\[
        \teta:G\times Y\lra L^{0},\qquad 
        \teta(g,y)=\fhi(g y)^{-1} f(\pi(y))^{-1}\,\alpha(g,\pi(y))\, f(\pi(y))\fhi(y).
\]
Our restriction on $f$ and the fact that $\fhi$ has finite range imply that
$\teta$ satisfies the integrability condition 
\begin{equation}
\tag{$L^{1}$}
        \int_{Y}\log\|\teta(g,y)\|\d\nu(y)\ < \infty\qquad(\forall\,g\in G).
\end{equation}
Let   $\ p:L^0\to H=L^{0}/\Ramen(L^{0})$ denote the quotient by the amenable radical. Then the cocycle
\begin{equation}
\label{e:pgamma}
        p\circ\teta:G\times Y\lra L^{0}\lra H    
\end{equation}
is a Zariski-dense in a connected centre-free semi-simple Lie group $H$ without 
compact factors.
This cocycle satisfies the ($L^{1}$) condition as well, and is almost ready for
an application of Theorem~\ref{T:cocycle1}.

\medskip

The only issue here is that the action of $G=G_{1}\times\cdots\times G_{n}$ on $(Y,\nu)$
might fail to be ergodically irreducible: although $G\acts (Y,\nu)$
is ergodic, this may not be true for the actions of the individual factors $G_{i}$.
However, since each $G_{i}$ acts ergodically on $(X,\mu)$, the $G_{i}$-action
on the finite extension $(Y,\nu)$ has a finite number of
ergodic components (at most $[L:L^{0}]$). For $i=1,\dots,n$ denote by 
\[
        P^{(i)}=(Y^{(i)}_{1},\dots,Y^{(i)}_{k_{i}})
\]
the partition of $(Y,\nu)$ into the $G_{i}$-ergodic components.
\begin{lem}\label{L:joint-part}
\ 
\begin{enumerate}
\item
\label{i:1}
        The partitions  $P^{(i)}$  are independent
        and the elements of the joint partition $P=P^{(1)}\vee\cdots\vee P^{(n)}$ are 
        transitively permuted by $G$; 
\item
\label{i:2}
        If $G^*$ denotes the stabiliser of an element, say $Z=\bigcap_{i=1}^{n} Y_{1}^{(i)}$, 
        of the joint partition $P$, then the map 
        \[
                q: Y\lra G/G^*, \quad q|_{gZ} \equiv gG^*
        \]
        is $G$-equivariant; 
\item
\label{i:3}
        The group $G^*$ is a direct product $G^*=G^*_{1}\times\cdots\times G^*_{n}$ of finite
        index subgroups $G^*_{i}<G_{i}$, and the action of each of $G^*_{i}$ on $Z$
        is ergodic.
\end{enumerate}
\end{lem}
\begin{proof}
We describe the case $n=2$, the general case following by induction.
Since $G$ acts ergodically on $(Y,\nu)$, the group $G_{2}\cong G/G_{1}$ permutes transitively
the collection $P^{(1)}$ of $G_{1}$-ergodic components. Hence, for each $1\le i\le k_{2}$,  the 
$G_{2}$-action transitively permutes the $k_{1}$ sets $Y^{(1)}_{j}\cap Y^{(2)}_{i}$, $1\le j\le k_{1}$.
In particular, each of these sets has size $\mu(Y^{(1)}_{j}\cap Y^{(2)}_{i})=\mu(Y^{(2)}_{i})/k_{1}$.
Similarly, $G_{1}$ transitively permutes the collection
$\{ Y^{(1)}_{j}\cap Y^{(2)}_{i} : 1\le i\le k_{2}\}$ and we obtain 
$\mu(Y^{(1)}_{j}\cap Y^{(2)}_{i})=\mu(Y^{(1)}_{j})/k_{2}$.
Thus $P^{(1)}\perp P^{(2)}$ and we proved claim~\eqref{i:1}. 

It further follows that the $G$-action on the elements of the joint
partition $P^{(1)}\vee P^{(2)}$ is through $G_{2}$ in the $P^{(1)}$-coordinate, 
and through $G_{1}$ in the $P^{(2)}$-coordinate.
Thus, the stabiliser of $Z=Y^{(1)}_{1}\cap Y^{(2)}_{1}$ has the form $G^*=G_{1}^*\times G_{2}^*$
with $[G_{1}:G_{1}^*]=k_{2}$ and $[G_{2}:G_{2}^*]=k_{1}$.
Therefore, the $G$-action on the elements of $P=P^{(1)}\vee P^{(2)}$ is through
$G/G^*$ as in~\eqref{i:2}.

Finally, to see the ergodicity of the $G^*$-action on $Z$, consider a positive measure subset 
$A\subset Z$ and a generic point $z\in Z$. 
Viewing $z$ and $A$ in $Y$, one can use the ergodicity of $G\acts (Y,\nu)$ to
deduce that $g z\in A$ for some $g\in G$.  
But since both $z$ and $g z$ are in $Z\subset Y^{(1)}_{1}$, the $G_{2}$-component
of $g\in G=G_{1}\times G_{2}$ is in $G_{2}^*$. Similarly, $z,g z\in Z\subset Y^{(2)}_{1}$ 
implies that the $G_{1}$-component of $g$ is in $G_{1}^*$. 
Hence we proved that for a generic $z\in Z$ and a positive measure $A\subset Z$
there is an element $g\in G^*=G_{1}^*\times G_{2}^*$ with $g z\in A$.
This proves the ergodicity claim in~\eqref{i:3}.
\end{proof}

We now return to the cocycle $p\circ\teta$ in~\eqref{e:pgamma}. Let $Z\subset Y$
and $G^*$ be as in Lemma~\ref{L:joint-part}.
Then the restriction $\delta: G^*\times Z\lra H$ of $p\circ \teta$ to $G^*\times Z$
satisfies the ($L^{1}$) condition, and the action $G^*\acts Z$ is ergodically irreducible.
The $G^*_i$-representation on $L^{2}_{0}(Z)$ is easily seen to 
inherit the spectral gap property from the $G_i$-representation on $L^{2}_{0}(X)$.
Hence we may apply Theorem~\ref{T:cocycle1} to deduce that $\delta$ is cohomologous
to the homomorphism $\ro^*:G^*\to H$ obtained in Section~\ref{S:s-s:hull}.
Using a measurable cross-section $H\to L^{0}<\GL_{d}(\RR)$ we may re-adjust
the initial conjugation map $f:Y\to \GL_{d}(\RR)$ to achieve a situation where
the cocycle $\tilde \beta$ lifted from $\beta$
\[
        \tilde\beta:G\times Y\lra L^{0}\qquad \tilde\beta(g,y)=f(g y)^{-1}\,\alpha(g,\pi(y))\,f(y)
\]
has the property that its restriction to $G^*\times Z$ projected \emph{via} $L^{0}\xrightarrow{p} H$
is the homomorphism $\ro^*:G^*\to H$
\[
        f(g z)^{-1}\,\alpha(g,\pi(z))\,f(z)=\ro^*(g)\qquad(g\in G^*<G,\ z\in Z\subset Y).
\]
Finally, the fact that the $G\acts Y$ has $G\acts G/G^*$ as a factor with $Z$ being
the preimage of $eG^*\in G/G^*$ means that the whole cocycle $p\circ \beta:G\times Y\to H$ 
is induced from the homomorphism $\ro^*:G^*\to H$ \emph{via} $Y\to G/G^*$ as claimed.
\end{proof}

\section{Lattices in Products}
\label{S:lattices}

In this section, we address Theorem~\ref{T:lattices}.
As a general fact for a lattice $\Gamma$ in a product group $G$, recall that it is equivalent to assume
property~(T) for $\Gamma$ or for $G$ or for all factors of $G$, see \emph{e.g.}~\cite{dHarpeValette:T:89}.

\medskip

\begin{proof}[Completion of the proof of Theorem~\ref{T:lattices}]
We have a cocycle $\alpha:\Gamma\times M\to \GL^1_d(\RR)$ verifying the~$(L^\infty)$ condition just
as in the proof of Theorem~\ref{T:product-of-T}.  Since $\Gamma$ is
finitely generated by property~(T)~\cite{Kazhdan67}, we may choose a word-length $\ell_\Gamma$, noting that what
follows will not depend on this choice.
The integrability assumption on $\Gamma$ means that there is a cocycle $c:G\times
G/\Gamma\to \Gamma$ in the canonical class such that
\begin{equation}\label{e:int-cocycle}
\int_{G/\Gamma} \ell_\Gamma (c(g, q))\d m_{G/\Gamma}(q)\ < \infty
\qquad(\forall\,g\in G).
\end{equation}
This integrability condition is known to hold for classical lattices (see \S~2
in~\cite{Shalom:Inven:00}) and for Kac--Moody groups viewed as lattices (see~\cite{Remy04}).

\smallskip

Let $L$ be the algebraic hull of $\alpha$, $L^0$ its neutral component and
$H= L^0/\Ramen(L^0)$. We record the following.

\begin{lem}\label{l:non-compact}
The group $H$ is non-compact.
\end{lem}

\begin{proof}
Otherwise, $L$ would be amenable since $L^0$ has finite index in $L$.
Since $\Gamma$ has property~(T), we recall that this would imply that
the hull $L$ is compact~\cite[9.1.1]{Zimmer:book:84}. In other words,
$\alpha$ could be conjugated into ${\rm O}_{d}(\RR)$, which means
that it preserves a measurable Riemannian structure on $M$.
Thus, $\Gamma\acts M$ would have discrete spectrum by Zimmer's 
result recalled in Section~\ref{S:T:product-of-T}, contradicting
the mixing assumption.
\end{proof}

We consider now the induced $G$-space $X=G/\Gamma \ltimes M$.

\begin{lem}\label{L:ind-erg}
Each $G_i$ acts ergodically and with spectral gap on $X$.
\end{lem}

\begin{proof}
The ergodicity of $G_i\acts G/\Gamma \ltimes M$ is equivalent to the ergodicity of $\Gamma$ on
$G/G_i \times M$ (Gel'fand--Fomin duality principle). Since $\Gamma$ is irreducible in $G$, it acts
ergodically on $G/G_i$. Thus the ergodicity on $G/G_i \times M$ is a well-known consequence of the
fact that $\Gamma\acts M$ is mixing (see \emph{e.g.}~\cite[3.7]{Stuck}).
The spectral gap follows from ergodicity by property~(T) of $G_i$.
\end{proof}

By property~(T) of $\Gamma$, any ergodic $\Gamma$-action has the spectral gap;
using Proposition~\ref{P:taming}, we can assume that some conjugate $\alpha^f$
of $\alpha$ ranges in $L$ and at the same time satisfies the~$(L^1)$ condition.
Consider the induced cocycle
$$\beta: G\times X\lra L, \qquad \beta(g, (q, x)) = \alpha^f(c(g, q), x).$$
It is a general fact that the operation of inducing cocycles does not change the algebraic hull,
see Lemma~3.1 in~\cite{Stuck}. Therefore:

\begin{lem}
The cocycle $\beta$ is Zariski-dense in $L$.\hfill\qedsymbol
\end{lem}

Further, we claim that induction using an integrable cocycle $c$ preserves integrability:

\begin{lem}
The cocycle $\beta: G\times X\to L$ satisfies the $(L^1)$ condition.
\end{lem}

\begin{proof}
Let $g\in G$ and decompose $G/\Gamma=\bigsqcup_{\gamma\in\Gamma} A_\gamma$ along
the fibres $A_\gamma$ of $c(g, -)$ over $\Gamma$. Thus~\eqref{e:int-cocycle}
can be written
$$\sum_{\gamma\in\Gamma} m_{G/\Gamma}(A_\gamma)\cdot \ell_\Gamma(\gamma)\ <\infty.$$
Let $\mu$ denote the probability measure on $M$ defined using the volume form.
Our choice of the conjugate $\alpha^f$ ensures
$$\ell(\gamma) := \int_X \log \|\alpha^f(\gamma, x)\|\d\mu(x)\ < \infty\qquad(\forall\,\gamma\in \Gamma).$$
Now the lemma follows from the identity
\begin{multline*}
\int_{G/\Gamma\times M} \log\|\beta(g, (q,x))\|\d(m_{G/\Gamma}\times \mu)(q,x)\\
=\sum_{\gamma\in\Gamma}\int_{A_\gamma}\int_X \log\|\alpha^f(\gamma,x)\|\d\mu(x)\d m_{G/\Gamma}=
\sum_{\gamma\in\Gamma} m_{G/\Gamma}(A_\gamma)\cdot \ell(\gamma)
\end{multline*}
together with the fact that $\ell$ is linearly controlled by $\ell_\Gamma$,
see Remark~\ref{R:word-length}.
\end{proof}

We now argue as in the proof of Theorem~\ref{T:cocycle2} and consider the $G$-equivariant finite
extension $\pi:Y\to X$ given by $Y=X\ltimes (L/L^0)$, which we recall is $G$-ergodic.
As we have seen in that proof, this provides us with a cocycle
$$\eta: G\times Y\lra L^0\lra H$$
which still retains the integrability and Zariski-density conditions established above for $\beta$.
We can therefore continue as in Section~\ref{S:cocycle2}.
Thus we have Zariski-dense representations of finite index subgroups $G_i^*<G$ to connected groups $H_i$
such that $H=\prod_i H_i$ and such that the product representation of $G^* = \prod_i G_i^*$ is conjugate
to $\eta|_{G^*}$. Upon possibly passing to finite extensions, this yields a representation of $G$ which is unbounded
by Lemma~\ref{l:non-compact}.
\end{proof}

\section{Additional Considerations}
\subsection{Geometric Approach to Cocycle Superrigidity}
\label{S:splitting}

Let $(M, g)$ be a compact Riemannian manifold. For each $x\in M$, consider the symmetric space 
$\GL(T_x M)/\mathrm{O}(g_x)$. The \emph{Pythagorean integral} (or \emph{induced space})
$$\int_X \GL(T_x M) / \mathrm{O}(g_x) \d x$$
is by definition the space of $L^2$-sections of this bundle (\cite[Ex.~47]{Monod:SuperSplitting:06})
and is a complete \cat0 space, indeed a Hilbert manifold of non-positive sectional curvature.
Volume-preserving diffeomorphisms of $M$ yield isometries of this space, and thus our original
motivation for investigating actions of product groups upon $M$ was the possibility to apply
to this situation the splitting theorem for general \cat0 spaces given in~\cite{Monod:SuperSplitting:06}.
Whilst we have prefered in this article to take advantage of the projective dynamics available in the much more
specific case of symmetric spaces, we shall nevertheless give below a very brief outline of this alternative approach.

\smallskip

The assumption of the splitting theorem is that the action on the \cat0 space is not \emph{evanescent},
which means that there should not be an unbounded set on which each group element has bounded displacement.
In order to reduce to that situation, one first replaces the above bundle by a smaller sub-bundle. More
specifically, we shall now sketch how to prove Theorem~\ref{T:cocycle1} following this geometric approach.
In view of the definition of the Pythagorean integral, we shall however replace the $L^1$-condition by the
slightly stronger $L^2$-condition (which holds in the situation arising from actions on compact manifolds).

\smallskip

As we have seen in Section~\ref{S:s-s:hull}, we may assume that our cocycle
$\alpha: G \times X\to H$, with $G=G_1\times G' _1$, is Zariski-dense when restricted to $G_1$
and we need to prove that its restriction to $G'_1$ cannot be Zariski-dense. In other words,
we shall explain how geometric splitting implies the key Proposition~\ref{P:H1H2H} for a simple
group $H$. In view of the cocycle reduction (Lemma~\ref{L:no-inv-in-Pm}), it essentially
suffices to derive a contradiction from the assumption that neither $G_1$ nor $G'_1$ admits an
equivariant measurable map from $X$ to probability measures on the geometric boundary of the symmetric space
$Y$ of $H$. Consider the induced $G$-space $S=\int_X Y$; the $L^2$-condition ensures that the obvious
isometric $G$-action is well-defined. The proof of Proposition~\ref{P:inv-in-Pm} can
be modified to yield the following statement for any proper \cat0 space $Y$:

\smallskip
\itshape
Let $L<G$ be a subgroup whose action on $X$ has the spectral gap. If the $L$-action on 
$S$ is evanescent, then there is an $\alpha|_L$-equivariant measurable map from $X$ to
probabilities on $\partial Y$.
\upshape

\smallskip

In particular, the $G$-action is non-evanescent and therefore the splitting theorem
(Theorem~9 in~\cite{Monod:SuperSplitting:06}) provides a canonical $G$-invariant subspace
$Z\se S$ with an isometric equivariant splitting $Z=Z_1\times Z'_1$ into $G_1$- and $G'_1$-spaces.

In fact, we shall use only a weaker statement which is a preliminary step in this splitting theorem,
namely the fact that $S$ contains a minimal (non-empty) $G'_1$-invariant \cat0 subspace $Z'_1$,
compare Proposition~35 in~\cite{Monod:SuperSplitting:06}.

\smallskip

We now have the following dichotomy. Either $Z'_1$ is bounded, in which case it is a point by minimality
(and the circumcentre lemma). Then $G'_1$ fixes a point in $S$, which means that $\alpha|_{G'_1}$ is conjugated
into a compact subgroup, contrary to our assumption. (Notice that this argument would not be possible if the simple
Lie group $H$ were allowed to be compact, compare Section~\ref{S:compact}.)
If on the other hand $Z'_1$ is unbounded, then it witnesses
the evanescence of the $G_1$-action on $S$. Indeed, minimality and convexity of the metric forces the displacement
lengths of elements of $G_1$ to be constant on $Z' _1$.
Applying the above statement to $L=G_1$, we have also a contradiction.

\smallskip

We observe that in the above outline of argumentation, just like in our main random walk argument, one needs only to assume
the spectral gap property for all but one factor (Remark~\ref{R:details}).

\subsection{Compact Targets}
\label{S:compact}
Finally, we explain why it is necessary in Theorem~\ref{T:cocycle1} to assume
that the semi-simple target group $H$ has no compact factors.
More specifically, the standard arithmetic construction below shows that the conclusions
of that theorem fail if the target group is a simple compact Lie group 
(compare with Proposition~\ref{P:H1H2H}).

\medskip

Fix some $n\ge 2$. Let $F=\QQ(\xi)$ be a totally real separable 
extension of $\QQ$ of degree $n+1$; we denote by $Gal(F/\QQ)$ the corresponding Galois
group and realize it as $Gal(F/\QQ)=\{\sigma_{0},\dots,\sigma_{n}\}$
where $\sigma_{i}:F\to \RR$ are Galois embeddings.
Upon replacing $\xi$ by a suitable rational translate,
one can assume that $\sigma_{0}(\xi)<0<\sigma_{1}(\xi),\dots,\sigma_{n}(\xi)$.
Let $D$ denote the diagonal matrix $D={\rm diag}[1,1,1,-\xi,-\xi]$
and consider the algebraic group $\mathbf{G}=\{ A\in \SL_{5} \mid  A^{T}DA=D\}$
defined over $F$. 
Under the Galois embeddings $\sigma_{i}$, the quadratic form defined 
by $D$ has signature $(3,2)$ for $1\le i\le n$, and is positive definite for $i=0$.
Denoting by $k_{i}$ the  Archimedean completions coming from $\sigma_{i}:F\to \RR$,
we get $G_{i}=\mathbf{G}(k_{i})\cong \SO_{3,2}(\RR)$ for $1\le i\le n$, 
and $K=\mathbf{G}(k_{0})\cong \SO_{5}(\RR)$. 
Let $\O_{F}$ denote the ring of integers of $F$. 
The group $\widetilde{\Gamma}=\mathbf{G}(\O_{F})$ embeds as a (uniform) lattice 
in the locally compact group
\[
        \widetilde{G}=\prod_{i=0}^{n} \mathbf{G}(k_{i})=K\times G_{1}\times\cdots\times G_{n}
\]
having dense injective projections on every sub-factor of the product 
\cite{BorelHarish-Chandra:62}. 
In particular, the projection $\tau:\widetilde{G}\to G=G_{1}\times\cdots\times G_{n}$
maps $\widetilde{\Gamma}$ to a lattice $\Gamma<G$, 
while $\pi:\widetilde{\Gamma}\to K$ is a dense embedding. 
Starting from a cocycle  $c:G\times G/\Gamma\to\Gamma$ in the canonical class, construct the cocycle
\begin{equation}\label{e:alpha}
        \alpha=\pi\circ \tau^{-1}\circ c:G\times G/\Gamma\to K.
\end{equation}
We claim that its restriction $\alpha|_{G_{i}}$ is Zariski
dense in $K\cong\SO_{5}(\RR)$ for each $1\le i\le n$. Since all the groups $G_{i}$
have Kazhdan's property~(T) which ensures the spectral gap assumption,
this claim indeed shows that Theorem~\ref{T:cocycle1} cannot hold for compact targets.

\smallskip

The proof of the claim relies on a well-known change of viewpoint
(which we have already used in an earlier section);
namely, it is equivalent to the ergodicity of the $G_i$-action
on the skew product $G/\Gamma\ltimes K$ associated to $\alpha$.
This latter action is isomorphic to 
\[
        G_{i}\acts \widetilde{G}/\widetilde{\Gamma}\ =\ (K\times G)/\widetilde{\Gamma}.
\]
We now recall that the following conditions are equivalent (Gel'fand--Fomin duality principle):
\begin{enumerate}
\item
        $G_{i}\acts \widetilde{G}/\widetilde{\Gamma}$ is ergodic;
\item
        $\widetilde{\Gamma}\acts \widetilde{G}/G_{i}=K\times G_{i}'$ is ergodic, where $G_{i}'=\prod_{j\neq i} G_{j}$;
\item
        $\widetilde{\Gamma}$ has a dense projection to $K\times G_{i}'$;
\item
        $\Gamma$ has a dense projection to $G_{i}'$.
\end{enumerate}
These conditions are satisfied by construction, proving the claim.

%

\end{document}